\documentclass[reqno]{amsart}

\usepackage[T1]{fontenc}	
\usepackage{graphicx}
\usepackage{amsmath, amsthm, amssymb}
\usepackage{upgreek}
\usepackage{mathrsfs}
\usepackage{pdfsync}
\usepackage[a4paper,top=3cm,bottom=3cm,left=2.5cm,right=2.5cm, bindingoffset=0mm]{geometry}
\usepackage[usenames,dvipsnames]{xcolor}
\usepackage[inline,shortlabels]{enumitem}
\usepackage[hyphens]{url}									
\usepackage{verbatim}								 
\usepackage{dsfont}									
\usepackage{cases}

\usepackage{mathtools}
\usepackage{xspace}
\usepackage{xparse}
\usepackage{microtype}

\usepackage{caption}
\usepackage{subcaption}

\definecolor{jan}{rgb}{0.0,0.3,0.8}

\def\EEE{\color{black}}

\usepackage[pdftex,bookmarks,
colorlinks,				
breaklinks,
citecolor=OliveGreen,
linkcolor=Maroon
]{hyperref}

\usepackage{mathtools}

\allowdisplaybreaks

\usepackage{xparse}

\ExplSyntaxOn

\NewDocumentCommand{\makeabbrev}{mmm}
{
	\yoruk_makeabbrev:nnn { #1 } { #2 } { #3 }
}

\cs_new_protected:Npn \yoruk_makeabbrev:nnn #1 #2 #3
{
	\clist_map_inline:nn { #3 }
	{
		\cs_new_protected:cpn { #2 } { #1 { ##1 } }
	}
}
\ExplSyntaxOff

\makeabbrev{\textbf}{tbf#1}{a,b,c,d,e,f,g,h,i,j,k,l,m,n,o,p,q,r,s,t,u,v,w,x,y,z,A,B,C,D,E,F,G,H,I,J,K,L,M,N,O,P,Q,R,S,T,U,V,W,X,Y,Z}

\makeabbrev{\textbf}{bf#1}{a,b,c,d,e,f,g,h,i,j,k,l,m,n,o,p,q,r,s,t,u,v,w,x,y,z,A,B,C,D,E,F,G,H,I,J,K,L,M,N,O,P,Q,R,S,T,U,V,W,X,Y,Z}

\makeabbrev{\textsf}{tsf#1}{a,b,c,d,e,f,g,h,i,j,k,l,m,n,o,p,q,r,s,t,u,v,w,x,y,z,A,B,C,D,E,F,G,H,I,J,K,L,M,N,O,P,Q,R,S,T,U,V,W,X,Y,Z}

\makeabbrev{\mathsf}{mss#1}{a,b,c,d,e,f,g,h,i,j,k,l,m,n,o,p,q,r,s,t,u,v,w,x,y,z,A,B,C,D,E,F,G,H,I,J,K,L,M,N,O,P,Q,R,S,T,U,V,W,X,Y,Z}

\makeabbrev{\mathfrak}{mf#1}{a,b,c,d,e,f,g,h,i,j,k,l,m,n,o,p,q,r,s,t,u,v,w,x,y,z,A,B,C,D,E,F,G,H,I,J,K,L,M,N,O,P,Q,R,S,T,U,V,W,X,Y,Z}

\makeabbrev{\mathrm}{mrm#1}{a,b,c,d,e,f,g,h,i,j,k,l,m,n,o,p,q,r,s,t,u,v,w,x,y,z,A,B,C,D,E,F,G,H,I,J,K,L,M,N,O,P,Q,R,S,T,U,V,W,X,Y,Z}

\makeabbrev{\mathbf}{mbf#1}{a,b,c,d,e,f,g,h,i,j,k,l,m,n,o,p,q,r,s,t,u,v,w,x,y,z,A,B,C,D,E,F,G,H,I,J,K,L,M,N,O,P,Q,R,S,T,U,V,W,X,Y,Z}

\makeabbrev{\mathcal}{mc#1}{A,B,C,D,E,F,G,H,I,J,K,L,M,N,O,P,Q,R,S,T,U,V,W,X,Y,Z}

\makeabbrev{\mathbb}{mbb#1}{A,B,C,D,E,F,G,H,I,J,K,L,M,N,O,P,Q,R,S,T,U,V,W,X,Y,Z}

\makeabbrev{\mathscr}{ms#1}{A,B,C,D,E,F,G,H,I,J,K,L,M,N,O,P,Q,R,S,T,U,V,W,X,Y,Z}

\makeabbrev{\mathrm}{#1}{
	Id,id,ran,rk,diag,stab,ann,conv,pr,ev,tr,End,Hom,sgn,im,op,can,fin,ext,red,tot,Leb,
	rot,usc,lsc,Lip,lip,bSymLip,osc,AC,loc,coz,z,
	supp,Opt,Adm,Cpl,Geo,GeoOpt,GeoAdm,GeoCpl,reg,res,graph,
	bd,co,Ric,Exp,dExp,dist,seg,Seg,cut,fcut,Cut,SDiff,Iso,Isom,diam,cl,Homeo,Diff,Der,vol,dvol,inj,relint, Graph, sub,
	var,law,Var,Poi,Gam,pa,so,iso,fs,inv,pqi,mix,erg,
	TestF,
}

\makeabbrev{\mathsf}{#1}{CD,BE,MCP,Ent,wMTW,MTW,Ch,RCD,EVI,Rad,dRad,SL,cSL,dSL,ScL,Irr,SC,wFe,VA}

\makeabbrev{\mathsc}{#1}{mmaf,cg}

\newcommand{\eps}{\varepsilon}

\newcommand{\mathsc}[1]{\text{\textsc{#1}}}

\newcommand{\forallae}[1]{{\textrm{\,for ${#1}$-a.e.~}}}

\newcommand{\dom}[1]{\mathrm{dom}(#1)}

\DeclareMathOperator{\eqdef}{\coloneqq}

\let\epsilon\varepsilon

\newcommand{\diff}{\mathop{}\!\mathrm{d}}					
\newcommand{\ddt}{\tfrac{\mathrm{d}}{\mathrm{d}t}}	
\newcommand{\ddu}{\tfrac{\mathrm{d}}{\mathrm{d}u}}

\newcommand{\tabs}[1]{\big\lvert#1\big\rvert}	
\newcommand{\abs}[1]{\left\lvert#1\right\rvert}

\newcommand{\set}[1]{\left\{#1\right\}}

\newcommand{\paren}[1]{\left(#1\right)}						
\newcommand{\tparen}[1]{\big({#1}\big)}

\newcommand{\tbraket}[1]{\big[#1\big]}

\newcommand{\tquadre}[1]{\big[#1\big]}

\newcommand{\seq}[1]{\paren{#1}}

\DeclareMathOperator{\emp}{\varnothing}
\newcommand{\N}{{\mathbb N}}
\newcommand{\R}{{\mathbb R}}

\newcommand{\slo}[2][]{\abs{\mrmD^-{#2}}_{#1}\!}
\newcommand{\slop}[3][]{\abs{\mrmD^-{#2}}_{#1}\!(#3)}
\newcommand{\comma}{\,\mathrm{,}\;\,}
\newcommand{\semicolon}{\,\mathrm{;}\;\,}
\newcommand{\fstop}{\,\mathrm{.}}

\usepackage{scrextend}

\let\temp\phi
\let\phi\varphi
\let\varphi\temp

\DeclarePairedDelimiter{\tond}{(}{)} 
\DeclarePairedDelimiter{\quadr}{[}{]}
\DeclarePairedDelimiter{\graf}{\{}{\}}

\DeclareMathOperator*{\argmin}{arg\,min}

\numberwithin{equation}{section}
\theoremstyle{plain}
\newtheorem{theorem}{Theorem}[section]

\newtheorem{lemma}[theorem]{Lemma}
\newtheorem{corollary}[theorem]{Corollary}

\theoremstyle{definition}
\newtheorem{definition}[theorem]{Definition}

\newtheorem*{defs*}{Definition}

\theoremstyle{remark}
\newtheorem{remark}[theorem]{Remark}

\newtheorem{ese}[theorem]{Example}
\newtheorem*{ass*}{Assumption}

\renewcommand{\paragraph}[1]{\medskip \emph{#1.}\ \ }

\def\tand{\quad{\rm and}\quad}

\begin{document}
	
	\title[Local conditions for global convergence]{Local conditions for global convergence\\of gradient flows and proximal point sequences\\in metric spaces}
	
	\author{Lorenzo Dello Schiavo}
	
	\author{Jan Maas}
	
	\author{Francesco Pedrotti}
	
	\address{IST Austria\\
		Am Campus 1\\
		3400 Klosterneuburg\\
		Austria
	}
	\email{lorenzo.delloschiavo@ist.ac.at}
	\email{jan.maas@ist.ac.at}
	\email{francesco.pedrotti@ist.ac.at}

	\begin{abstract}
		This paper deals with local criteria for the convergence to a global minimiser for gradient flow trajectories and their discretisations.
		To obtain quantitative estimates on the speed of convergence, we consider variations on the classical Kurdyka--{\L}ojasiewicz inequality for a large class of parameter functions. 
		Our assumptions are given in terms of the initial data, without any reference to an equilibrium point. 
		The main results are convergence statements for gradient flow curves and proximal point sequences to a global minimiser, together with sharp quantitative estimates on the speed of convergence.
		These convergence results apply in the general setting of lower semicontinuous functionals on complete metric spaces, generalising recent results for smooth functionals on $\R^n$. While the non-smooth setting covers very general spaces, it is also useful for (non)-smooth functionals on $\R^n$.
	\end{abstract}
	
	
	
	\maketitle

	\begin{center}
		\today
	\end{center}
	
	\setcounter{tocdepth}{1} 
	\tableofcontents

	\section{Introduction}
	
	For given $x_0 \in \R^n$ and $f \in \mrmC^2(\R^n)$ we  consider the 
	\emph{gradient flow} equation  
	\begin{align}
		\label{eq:grad-flow}
		\ddt y_t = - \nabla f(y_t)\comma
		\qquad
		y_0 = x_0 \fstop
	\end{align}
	It is of great interest in many applications to find conditions which guarantee convergence of  
	gradient-flow trajectories 
	$\seq{y_t}_{t \geq 0}$ 
	to a global minimiser of $f$ as $t \to \infty$,
	and to quantify the speed of convergence.
	This also applies to the associated discrete-time schemes, such as
	\emph{gradient descent} (or \emph{forward Euler}), the discrete-time scheme with step-size~$\tau > 0$ given by
	\begin{align}
		\label{eq:grad-descent}
		y_{k+1} = y_k - \tau \nabla f(y_k)\comma 
		\qquad
		y_0 = x_0 \comma
	\end{align}
	and the \emph{backward Euler} scheme
	\begin{align}
		\label{eq:back-euler}
		y_{k+1} = y_k - \tau \nabla f(y_{k+1})\comma 
		\qquad
		y_0 = x_0 \fstop
	\end{align}

	\subsection*{The Polyak--{\L}ojasiewicz condition}
	A very simple celebrated criterion for convergence to a global minimum is the 
	\emph{Polyak--{\L}ojasiewicz condition} \cite{Pol-1963},
	which requires neither the uniqueness of a minimiser nor the convexity of the function~$f$.
	The condition holds if, for some $\beta > 0$, 
	\begin{align*}
		|\nabla f(x)|^2 
		\geq 
		\beta (f(x) - f^\star)\comma \qquad x \in \R^n\comma
		\tag{P\L}
	\end{align*}
	where $f^\star$ is the global minimum of $f$, which is assumed to be attained.
	Since $\ddt f(y_t) = - |\nabla f(y_t)|^2$ 
	along any solution $y_t$ to the gradient-flow equation 
	$\ddt y_t = - \nabla f(y_t)$, 
	an application of Gronwall's inequality yields
	the exponential bound
	\begin{align*}
		f(y_t) - f^\star
		\leq
		e^{-\beta t} 
		\big( 
		f(y_0) - f^\star
		\big) \comma \qquad t \geq 0 \fstop
	\end{align*}
	Moreover, a short argument shows that 
	$y_t$ converges to a global minimiser $x^\star$, 
	with the bound 
	\begin{align*}
		|y_t - x^\star|^2
		\leq 
		\frac{4}{\beta} \big(f(y_t) - f^\star\big) \comma \qquad t\geq 0\fstop
	\end{align*}
	These inequalities together yield exponentially fast convergence to $x^\star$. 
	Analogous results hold for the associated gradient-descent scheme \eqref{eq:grad-descent} and for certain proximal-gradient methods   
	\cite{Karimi-Nutini-Schmidt:2016}. 
	Interestingly, in spite of its simplicity, it 
	has been argued in  
	\cite{Karimi-Nutini-Schmidt:2016} that the 
	P\L\mbox{}  condition 
	\guillemotleft\emph{is actually weaker than the 
		main conditions that have been explored to show linear convergence rates without strong convexity
		over the last 25 years.}\guillemotright

	\subsection*{The Kurdyka--{\L}ojasiewicz condition}
	An important generalization of the 
	{P\L} condition is the 
		\emph{Kurdyka--{\L}ojasiewicz inequality} (K{\L}), 
	that was introduced by {\L}ojasiewicz \cite{Loj-1963,Loj-1993} and later generalized by Kurdyka~\cite{Kur-1998}.

	\begin{definition}\label{def:KL:intro}
		Let $\theta\in  C([0,\infty)) \cap C^1((0, \infty))$
		satisfy 
			$\theta(0) = 0$ 
		and 
			$\theta'(u) > 0$ for $u > 0$.
		We say that the K{\L} inequality is satisfied in a neighbourhood $U$ of an equilibrium point 
			$x^\star \in \R^n$ 
		if
		\begin{align}
			\label{it:Kl:intro} 
			\theta'(f(x) - f(x^\star)) \cdot 
			|\nabla f(x)| \geq 1 \qquad\text{ for all } \; x\in U \cap \set{f > f(x^\star)}.
		\end{align}
	\end{definition}
	In applications, $\theta$ is often of the form
		$\theta(u) \coloneqq \frac{c}{\gamma} u^\gamma$ 
	with $\gamma \in (0,1]$ and $c>0$.
	In this case, \eqref{it:Kl:intro} reads as
	\[
	c \,|\nabla f(x)| \geq \tparen{f(x) - f(x^\star)}^{1-\gamma}\comma \qquad  x\in U \cap \set{ f > f(x^\star)}\fstop
	\]
	In particular, if 
		$\gamma = \frac12$ and $c = 1/\sqrt{\beta}$
	one recovers the {P\L} inequality.

	The K{\L} condition is a powerful tool to obtain convergence properties for gradient-flow solutions 
	and discrete schemes. 
	An important feature of the {K\L} condition is that the inequality is only required to hold \emph{locally}, on a suitable neighbourhood $U$ of an equilibrium point. 

	To obtain convergence results
	for gradient-flow trajectories to an equilibrium point,
	the {K\L} condition is often combined with 
	additional information,
	typically an upper bound 
	on the length of the trajectory,
	to ensure that the solution is eventually contained in~$U$;
	cf.~\cite{Sim-1983, Hauer-Mazon:2019} 
	for results of this type for gradient flows and 
	\cite{AttBol2009, AttBolRedSou2010, Att-Bol-Sva-2013, Bol-Dan-Ley-Maz-2010} 
	for discrete schemes. 
	
	Let us also remark that the {K\L} condition does not in general yield convergence to a global minimiser of~$f$, but merely to a stationary point.
	To deduce convergence to a global minimiser, 
	it is often required to know 
	\emph{a priori} 
	that the starting point is close enough to this minimiser
	(whose existence is often part of the assumption); 
	cf.~\cite[Thm.~10]{AttBolRedSou2010} and \cite[Thm.~2.12]{Att-Bol-Sva-2013} for such results for discrete schemes.

	\subsection*{A {P\L} condition around the starting point}
	A remarkable variant of the {P\L} condition was discussed by Oymak and Soltanolkotabi \cite{pmlr-v97-oymak19a}
	and 
	by 
	Chatterjee~\cite{Cha22} 
	for \emph{non-negative} functions~$f \in \mrmC^2(\R^n)$.
	For fixed $x_0 \in \R^n$, these authors consider the local quantity  
	\begin{align*}
		\alpha = \alpha(x_0, r) 
		\eqdef 
		\inf_{\substack{x\in B_r(x_0)\\f(x)\neq 0} }
		\frac{|\nabla f(x)|^2}{f(x)}\comma
	\end{align*}
	where $B_r(x_0)$ denotes the open ball of radius $r > 0$ around $x_0$.\footnote{In a general metric space, the closure~$\overline{B_r(x_0)}$ is a subset of the closed  ball~$\set{x\in X: \mssd(x,x_0)\leq r}$ and the inclusion may be strict.}
	The criterion in \cite{Cha22} requires that
	\begin{align}\label{eq:Chatterjee-intro}
		\alpha(x_0, r) > \frac{4 f(x_0)}{r^2}
	\end{align}
	for some $x_0 \in \R^n$ and some $r > 0$.
	In other words,
	the inequality $|\nabla f(x)|^2 \geq \beta f(x)$
	is imposed to hold for all $x \in B_r(x_0)$,
	with a sufficiently large constant $\beta$,
	namely, $\beta > 4 f(x_0) / r^2$.

	Under \eqref{eq:Chatterjee-intro}, 
	it is shown in~\cite{Cha22} that the unique gradient flow curve 
	$(y_t)_{t \geq 0}$ starting at $y_0 = x_0$ 
	stays within $B_r(x_0)$ for all times $t \geq 0$, 
	converges to a global minimiser 
	$x^\star \in \overline{B_r(x_0)}$, 
	and satisfies the exponential bounds
	\begin{align}
		\label{eq:Chatterjee-cont}
		f(y_t) 
		\leq
		e^{-\alpha t} 
		f(x_0) 
		\tand
		|y_t - x^\star|^2
		\leq 
		r^2 e^{-\alpha t} 
	\end{align}
	for $t \geq 0$, 
	where we write $\alpha = \alpha(x_0, r)$ for brevity.

	Like the {K\L} condition, \eqref{eq:Chatterjee-intro} is a local version of the {P\L} condition.  
	However, while the {K\L} condition involves information of $f$ in a neighbourhood of an equilibrium point (whose location is often unknown in applications), 
	\eqref{eq:Chatterjee-intro} is formulated in terms of the starting point $x_0$ of the gradient-flow trajectory.
	The existence of a global minimiser and the boundedness of the gradient-flow trajectory are not assumed; these statements are part of the conclusion. 
	The specific constant $\frac{4 f(x_0)}{r^2}$ in \eqref{eq:Chatterjee-intro} is important, as it ensures that the gradient-flow curve does not leave the ball $B_r(x_0)$, so that local information suffices to draw conclusions on the long-term behaviour.

	Chatterjee also proves analogous bounds for the gradient descent \eqref{eq:grad-descent} starting at $y_0 = x_0$, 
	namely
	\begin{align*}
		f(y_k) \leq (1 - \delta)^k f(x_0)
		\tand
		|y_k - x^\star| \leq r^2 ( 1 - \delta)^k  
	\end{align*}
	for all $k \in \N$ and any $\delta < \alpha \tau$,
	provided that the step-size $\tau > 0$ is sufficiently small, depending also on the size of the derivatives of $f$ in $B_r(x_0)$.
	Similar results for gradient descent were obtained previously in \cite{pmlr-v97-oymak19a}.
	Applications of \eqref{eq:Chatterjee-intro} to neural networks can be found in 
	\cite{pmlr-v97-oymak19a,Boursier-Pillaud-Vivien-Flammarion:2022,
	Cha22,
	Bombari:2022aa,
	Bartlett:2021aa}.

	\subsection{Main results}
	In this work we generalise some of the results of \cite{Cha22} in the following ways.
	First, we replace $\mathrm{C}^2$~functions on~$\R^n$ with lower semicontinuous functionals on complete metric spaces.
	Secondly, we replace the local P{\L}-like condition with a K{\L}-like assumption with a more general parameter function $\theta$.
	Thirdly, we prove convergence results for the proximal point method, which corresponds to the backward Euler scheme in smooth settings.

	Let~$(X,\mssd)$ be a complete metric space. In this generality, the \textsc{ode} \eqref{eq:grad-flow} does not have a direct interpretation, as the velocity of a curve and the gradient of a function are not defined. 
	However \eqref{eq:grad-flow} admits an equivalent variational characterisation, as a \emph{curve of maximal slope}, and this notion naturally extends to metric spaces.
	We refer to Section \ref{sec:CMS} for the definition of the metric slope $\slo{f}(x)$ and other concepts from analysis in metric spaces relevant to our work.
	Gradient flows in metric spaces are ubiquitous in applications; 
	notable examples are
	dissipative \textsc{pde}s 
	in the Wasserstein space 
	\cite{JorKinOtt98,AmbGigSav14} 
	and related gradient flows 
	on spaces of (probability) measures; see, e.g.,
	\cite{Dolbeault-Nazaret-Savare:2009,Maa11,Mielke:2011,Kondratyev-Vorotnikov:2018}.
	A systematic treatment can be found in the monograph~\cite{AmbGigSav08}.
	The metric point of view can also be useful to deal with non-differentiable functionals on $\R^n$; see \S\ref{sec:assumption} for some toy examples.
	
	We first define the functions appearing in the assumption and the main results.
	
	\begin{definition}[Parameter function]\label{d:Parameter1}
		We say that~$\theta\in C^1((0, \infty)) \cap C([0,\infty))$ is a \emph{parameter function} if
			 $\theta'(u) > 0$ for $u > 0$,
			 and 
			$\theta(0) = 0$.
		Furthermore, we consider the auxiliary functions $\eta: [0,\infty)  \to [-\infty,\infty)$
			and  $\Gamma : [0,\theta(\infty))  \to [-\infty,\infty)$ 
		defined by
		\begin{align*}
			\eta(u)
			 :=
			\int_1^u
			\big(\theta'(s)\big)^2
			\diff s
		\tand
		\Gamma(u)  
			:= (\eta \circ \theta^{-1})(u).
		\end{align*}
	\end{definition}

	The next definition contains a generalisation of \eqref{eq:Chatterjee-intro} to the metric setting for a general class of parameter functions; see Remark 
	\ref{rmk:comp:cha} below for a precise comparison. 
	
	\begin{definition}[Conditions {$(A)$} and {$(A')$}]\label{def:main-ass}
		For~$x_0\in\dom{f}$ and~$r>0$, we say that condition~{$(A)$} is satisfied with parameter function~$\theta$ (as in Definition \ref{d:Parameter1}) if
		\begin{equation}
			\label{eq:ass:cts}
			(\theta\circ f)(x_0)\leq r 
			\quad \text{and} \quad
			(\theta'\circ f)(x)\cdot \slo{f}(x)\geq 1 \comma \quad x \in B_r(x_0) \cap \set{0<f\leq f(x_0)}\fstop
		\end{equation}
		Similarly, we say that condition~{$(A')$} is satisfied if
		the first inequality in 
		\eqref{eq:ass:cts} 
		is replaced by the strict inequality 
		$(\theta\circ f)(x_0)< r$.
	\end{definition}

	Under Condition~{$(A)$}, 
	our first main result asserts that
	gradient-flow trajectories stay within in a bounded set and converge to a global minimum, with
	quantitative bounds on the rate of convergence.

	\begin{theorem}[Convergence of gradient flows]\label{thm:ConvRateIntro}
		Let $f : X \to [0, \infty]$ be proper and lower semicontinuous,
		and suppose that 
			$x_0\in\dom{f}$ and~$r>0$ 
		satisfy Condition~{$(A)$} 
		for some parameter function $\theta$.
		For some $T\in (0,\infty]$,
		let $\seq{y_t}_{t\in [0,T)}$ be a curve of maximal slope for~$f$ starting at~$x_0$. 
		Then:
		\begin{enumerate}[$(i)$]
			\item\label{i:thm-conv-1} \emph{(confinement)} $y_t \in \overline{B_r(x_0)}$
			for all $0\leq t<T$. Moreover, $y_t \in B_r(x_0)$ for all $0\leq t<T$ with $f(y_t) > 0$. 
			\item\label{i:thm-conv-2}  
			\emph{(convergence)} $y_T\eqdef \lim_{t\to T} y_t$ exists and belongs to $\overline{B_r(x_0)}$. Moreover, 
			\begin{align}\label{eq:theta-estimate-thm}
				(\theta\circ f)(y_s)- (\theta\circ f)(y_t)  \geq 
				\mssd(y_t,y_s)
			\end{align}
			for all $0\leq s\leq t \leq T$.
			In particular, $y_T\in B_r(x_0)$ if $f(y_T) > 0$. 
			\item \label{i:thm-conv-3} \emph{(convergence rates)}
				Set
			$t_*
			\coloneqq \inf\set{t\in [0,T): f(y_t)=0} \wedge T$.
			The following bounds hold for $0 \leq t \leq t_*$:
			\begin{align}
				\label{eq:d-rate}
				\Gamma\big(\mssd(y_t,y_T)\big)& \leq \Gamma(r) - t, 
				\\
				\label{eq:f-rate}
				(\eta\circ f)(y_t)		
				& \leq 
				(\eta\circ f)(x_0)
				- t.
			\end{align}
			Moreover, if $T=\infty$ then $f(y_\infty) = 0$.
		\end{enumerate}
	\end{theorem}

	In the special case of Remark \ref{rmk:comp:cha}, the previous theorem yields the following generalisation of~\cite[Thm.~2.1]{Cha22} to the setting of metric spaces; see also Cor.~{\ref{cor:conv-rate-gf-special}} below for a version with more general parameter functions. 
	For $x_0 \in \dom f$ and $r > 0$
	we define
	\begin{align*}
		\alpha = \alpha(x_0, r) 
			\eqdef 
			\inf_{\substack{x\in B_r(x_0)\\0<f(x)\leq f(x_0)} }
			\frac{\slop{f}{x}^2}{f(x)} \, .
	\end{align*}

	\begin{corollary}
		\label{cor:chatt-analogue-cts}
		Let $f : X \to [0, \infty]$ be proper and lower semicontinuous,
		and suppose that 
		$\alpha(x_0,r) \geq 4 f(x_0) / r^2$
		for some 
		~$x_0\in\dom{f}$ and~$r>0$.
		For some $T\in (0,\infty]$,
		let $\seq{y_t}_{t\in [0,T)}$ be a curve of maximal slope for~$f$ starting at~$x_0$. 
		Then 
		$y_T\coloneqq \lim_{t\to T} y_t$ exists, 
		$y_t$ belongs to $\overline{B_r(x_0)}$ 
		for all  $t \in [0,T]$,
		and
		\begin{align*}
			\mssd(y_t,y_T) \leq r\, e^{-\alpha t/2} 
			\qquad \text{and} \qquad	
			f(y_t) \leq e^{-\alpha t} f(x_0)
		\end{align*}
		for all~$t \in [0, T ]$, 
		where, conventionally,~$e^{-\infty}\eqdef 0$. 
	\end{corollary}

Various works deal with convergence of gradient-flow trajectories under a K{\L} condition in the setting of metric spaces \cite{Blanchet-Bolte:2018,Hauer-Mazon:2019}; see also \cite{Bol-Dan-Ley-Maz-2010,AttBol2009, AttBolRedSou2010} for related work on proximal point sequences. 
Applications have been found to convergence of mean-field birth-death processes \cite{Liu-Majka-Szpruch:2022} 
and to swarm gradient dynamics \cite{Bolte-Miclo-Villeneuve:2022}.

The main estimates in our paper are obtained by adapting known arguments from,
	e.g.,~\cite{Blanchet-Bolte:2018,Hauer-Mazon:2019}.
	However, as in \cite{pmlr-v97-oymak19a,Cha22}, our point of view differs from these works, as we work under a local condition in terms of the starting point without referring to an equilibrium point in the assumption.

	\subsubsection*{Discrete schemes}
	
	In the general setting of metric spaces, we are not aware of any way to formulate a forward Euler scheme \eqref{eq:grad-descent}.
	However, the backward Euler scheme admits an equivalent metric formulation as 
	a minimising movement scheme (or proximal point method).
	This scheme was originally introduced by Martinet \cite{Mar1970} and Rockafellar \cite{Roc1976} as a natural regularisation method in optimisation problems:
	\begin{align*} \label{eq:ProximalPoint}
		y_{k+1} \in 
		\argmin_{x \in X}
		\bigg\{ f(x) + \frac{1}{2\tau}
		\mssd(y_k,x)^2
		\bigg\}\comma
		\qquad
		y_0 = x_0\fstop
	\end{align*}
	Any (finite or infinite) sequence $\seq{y_k}_{k }$ arising in this way is called a \emph{proximal point sequence} (or \emph{$\tau$-minimising movement sequence}).

	Our second main result is an analogue of Theorem \ref{thm:ConvRateIntro} 
	for the proximal point method, under the slightly stronger assumption {$(A')$}.

	\begin{theorem}\label{t:conv-disc-new}
		Let $f : X \to [0, \infty]$ be proper and lower semicontinuous,
		and suppose that 
			$x_0\in\dom{f}$ and~$r>0$ 
		satisfy Condition~{$(A')$} 
		for some parameter function $\theta$.
		Suppose further that there exists 
		$\bar{\tau} > 0$ 
		such that, for all 
		$x \in B_r(x_0)\cap \graf*{f\leq f(x_0)}$ 
		and $\tau \in (0, \bar \tau)$, 
		the functional
		\begin{align*}
			X \ni y \longmapsto f(x) + \frac{1}{2 \tau} \mssd(x,y)^2
		\end{align*}
		has at least one global minimiser.
		Then
		there exists an infinite 
		proximal point sequence
		starting from $x_0$,
		for any step-size
		$\tau < \bar{\tau}$. 
		Moreover, for any such sequence $(y_k)_{k=0}^\infty$, 
		the following statements hold:
		\begin{enumerate}[$(i)$]
			\item \emph{(confinement)}  
			\label{it: confinement discrete scheme} 
			$y_k \in B_r(x_0)$ for all $k \geq 0$;
			\item \emph{(convergence)}
			\label{it: convergence discrete scheme} 
			$y_\infty \eqdef \lim_{k\to \infty} y_k$ 
			exists
			and belongs to $B_r(x_0)$.
			Moreover, 
			$f(y_\infty)=0$;
			\item \emph{(distance bound)}
			\label{it: dist bound discrete} 
			For all $0\leq i\leq j < \infty$ we have
			\begin{align}
				\mssd(y_i, y_j) \leq (\theta\circ f)(y_i) - (\theta\circ f)(y_j)
					\quad\text{and} \quad
					\mssd(y_i, y_\infty) \leq (\theta\circ f)(y_i) \fstop
			\end{align}
		\end{enumerate}
	\end{theorem}

	In the particular case where $\theta$ takes the form $\theta(u) \coloneqq 2c\sqrt{u}$, we obtain the following result. 
	In this case we also obtain an estimate for the speed of convergence of $f(y_k)$ to $0$.
	Other special cases of Theorem \ref{t:conv-disc-new} are presented in Corollary \ref{cor:conv-rate-power-dis} below.

	\begin{corollary}[see Cor. \ref{cor:conv-rate-power-dis}]\label{cor-disc-square-root}
		Let $f : X \to [0, \infty]$ be proper and lower semicontinuous,
		and suppose that 
		and suppose that 
		$\alpha(x_0,r) > 4 f(x_0) / r^2$
		for some 
		~$x_0\in\dom{f}$ and~$r>0$.
		Suppose further that there exists 
		$\bar{\tau} > 0$ 
		such that, for all 
		$x \in B_r(x_0)\cap \graf*{f\leq f(x_0)}$ 
		and $\tau \in (0, \bar \tau)$, 
		the functional
		\begin{align*}
			X \ni y \longmapsto f(x) + \frac{1}{2 \tau} \mssd(x,y)^2
		\end{align*}
		has at least one global minimiser.
		Then
		there exists an infinite 
		proximal point sequence
		starting from $x_0$,
		for any step-size
		$\tau < \bar{\tau}$. 	
		Moreover, for any such sequence, 
		the following statements hold:
		\begin{enumerate}[$(i)$]
			\item \emph{(confinement)}  
			\label{it: confinement discrete scheme cor} 
			$y_k \in B_r(x_0)$ for all $k \geq 0$;
			\item \emph{(convergence)}
			\label{it: convergence discrete scheme cor} 
			$y_\infty \eqdef \lim_{k\to \infty} y_k$ 
			exists
			and belongs to $B_r(x_0)$.
			Moreover, 
			$f(y_\infty)=0$;
			\item \emph{(convergence rates)}
			\label{it: convergence rates discrete scheme} 
			The following bounds hold for all $k \geq 0$:
			\begin{align*}
				f(y_k) 
				& \leq 
				\big( 1 + {\alpha \tau} \big)^{-k} f(x_0)
				 \tand
				\mssd(y_k,y_\infty) 
				\leq 
				\big( 1 + {\alpha \tau} \big)^{-k/2} \, r \fstop
			\end{align*}
		\end{enumerate}
	\end{corollary}

\subsection*{Plan of the work}
Preliminaries on gradient flows in metric spaces are collected in~\S\ref{sec:CMS}.
Our main results in the continuum case are proved in \S\ref{sec:ConvGF} and extended to piecewise gradient-flow curves in~\S\ref{sec:Extension}.
In~\S\ref{sec:assumption} we discuss Condition {$(A)$} and its variants together with some examples.
Our main results in the discrete case are proved in~\S\ref{sec: convergence of discrete scheme}.
Auxiliary results are presented in the appendices.

	\subsection*{Acknowledgement} {\small 
		The authors gratefully acknowledges support by the European Research Council (ERC) under the European Union's Horizon 2020 research and innovation programme (grant agreement No.~716117). 
		LDS gratefully acknowledges funding of his current position by the Austrian Science Fund (FWF), ESPRIT Fellowship Project 208.
		JM also acknowledges support by the Austrian Science Fund (FWF), Project SFB F65.
	}

	\section{Gradient flows in metric spaces}\label{sec:CMS}
	In this section we collect some known facts about gradient flows in metric spaces.
	We assume throughout that ~$(X,\mssd)$ is a complete metric space and $J \subset \R$ is a (not necessarily open, nor closed) interval.
	
	Let $1 \leq p < \infty$. 
	A measurable function $m  : J \to \R$ belongs to  ~$L^p_\loc(J)$ if ${\bf 1}_K m \in L^p(J)$ for every \emph{compact} set~$K\subset J$.
	A curve~$\seq{y_t}_{t\in J}$ is said to be locally $p$-absolutely continuous on~$J$ ---in short: it belongs to~$\AC^p_\loc(J;X)$--- if there exists~$m\in L^p_\loc(J)$ so that
	\begin{equation}\label{eq:MSpeed}
		\mssd(y_t,y_s)
			\leq 
		\int_s^t m(r)\diff r 
	\end{equation}
	for all $s,t \in J$ with $s < t$.
	Similarly, we write ~$\seq{y_t}_{t\in J} \in \AC^p(J;X)$ if~$m\in L^p(J)$.
		
	Whenever~$\seq{y_t}_{t\in J}$ is in~$\AC^1_\loc(J;X)$, the \emph{metric speed}
	\begin{align}
		\label{eq:metric-speed}
		\abs{\dot y_t}\eqdef \lim_{s\to t} \frac{\mssd(y_s,y_t)}{\abs{s-t}}
	\end{align}
	exists for a.e.~$t \in J$.
	Furthermore, the metric speed coincides a.e.~with the smallest function ~$m$ satisfying~\eqref{eq:MSpeed}; see, e.g.,~\cite[Thm.~1.1.2]{AmbGigSav08}.
	
	\begin{remark}
		Every curve in $\AC^1_\loc(J;X)$ is continuous on $J$.
		Note however that if $\seq{y_t}_{t\in J}\in \AC^1_\loc(J;X)$ with ~$J=(a,b]$ for some~$a<b$, then the existence of~$ 
		\lim_{t\downarrow a} y_t$ does not imply that~$\seq{y_t}_{t\in J} \in \AC^1(J;X)$.
	\end{remark}

	\medskip
	
	The \emph{domain} of a function 
	$f\colon X \to (-\infty, \infty]$ 
	is the set
	\[
	\dom{f}\eqdef\set{x\in X: f(x)<\infty}\fstop
	\]
	In order to rule out trivial statements, we always assume that~$f$ is \emph{proper}, i.e.,~$\dom{f}\neq \emp$.
	The (\emph{descending}) \emph{slope} of~$f$ at~$x \in X$ is the quantity
	\[
	\slo{f}(x)\eqdef \limsup_{y\to x} \frac{\tbraket{f(y)-f(x)}_-}{\mssd(y,x)} \comma
	\]
	where~$a_- \eqdef \max\set{-a,0}$ denotes the negative part of $a \in \R$.
	Conventionally,~$\slo{f}(x)\eqdef 0$ when~$x \in \dom{f}$ is isolated, and~$\slo{f}(x)=+\infty$ if~$x\notin \dom{f}$.
	
	\subsection{Gradient flows in metric spaces: curves of maximal slope}
	The next definition provides a natural notion of gradient flow in a metric space; cf.~\cite{AmbGigSav08} for an extensive treatment. 
	The motivation for this definition comes from the following simple argument in Euclidean space. 
	Let $f\colon \R^n \to \R$ be a smooth function. For any smooth curve 
	$(u_t)_{ t \in [0,T)}$ in $\R^n$ and $t \in (0,T)$, 
	we have
	\begin{align*}
		- \ddt f(u_t)
		= 
		-	\nabla f(u_t) \cdot \dot u_t
		\leq 
		\tfrac12 |\nabla f(u_t)|^2
		+ 	\tfrac12 |\dot u_t|^2 \fstop
	\end{align*}
	Since equality holds if and only if $\dot u_t = - \nabla f(u_t)$, the reverse inequality 
	$- \ddt f(u_t)
	\geq 
	\frac12 |\nabla f(u_t)|^2
	+ \frac12 |\dot u_t|^2$
	is an equivalent formulation of the gradient-flow equation, which admits a natural generalisation to metric spaces.

	\begin{definition}[curve of maximal slope, gradient flow, cf.~{\cite[Dfn.~4.1]{MurSav20}}]\label{d:CMS}
		Let~$J\subset \R$ be an interval and let 
		$f : X \to (-\infty, \infty]$
		be proper.
		We say that~$\seq{y_t}_{t\in J}$ is a \emph{curve of maximal slope 
			for~$f$} if
		\begin{enumerate}[$(a)$]
			\item\label{i:d:CMS:1} $\seq{y_t}_{t \in J} \in \AC^1_\loc(J;X)$;
			\item\label{i:d:CMS:2} $\seq{f(y_t)}_{t \in J} \in \AC^1_\loc(J;\R)$;
			\item\label{i:d:CMS:3} the following \emph{Energy Dissipation Inequality} holds:
			\begin{equation}\tag{$\mathsc{edi}$}\label{eq:EDI}
				- \ddt f(y_t) 
				\geq
				\tfrac{1}{2} \abs{\dot y_t}^2
				+ \tfrac{1}{2}\slo{f}(y_t)^2 \quad \forallae{\diff t} t\in J\fstop
			\end{equation}
		\end{enumerate}
		If additionally ~$y_{\inf J}\eqdef \lim_{t\downarrow \inf J} y_t$ exists, we say that~$\seq{y_t}_{t\in J}$ is a curve of maximal slope for~$f$ \emph{starting at~$y_{\inf J}$}.
	\end{definition}
	
	\begin{remark}
		From~\eqref{eq:EDI} and the absolute continuity of $f$ we conclude that~$t\mapsto f(y_t)$ is non-increasing along curves of maximal slope~$\seq{y_t}_{t\in [0,T)}$. 
	\end{remark}

		There are several slightly different notions of curve of maximal slope in the literature, and the distinction matters for our purposes.
		In particular, it is important here to include the absolute continuity of the function along gradient-flow trajectories in our definition, as this allows one to 
		deduce the following well-known fact, asserting the equality of the speed of the gradient flow and the slope of the driving functional.

	\begin{lemma}\label{l:EqualityEDI}
		Let 
		$f : X \to (-\infty, \infty]$
		be proper
		and let
		$\seq{y_t}_{t\in J}$ 
		be a  curve of maximal slope.
		For \textrm{a.e.} $t \in J$ we have 
		\begin{align}\label{eq:EqYoung}
			- \ddt f(y_t)  
			= \abs{\dot y_t}^2
			= \slop{f}{y_t}^2.
		\end{align}
		In particular, equality holds for a.e.~$t\in{J}$
		in \eqref{eq:EDI}.
	\end{lemma}	
	
	\begin{proof}
		Let $t \in J$ be such that 
		the metric speed
		$\abs{\dot y_t}$
		and the derivative
		$\ddt f(y_t)$ 
		exist. 
		By local absolute continuity of $(y_t)_t$ and $(f(y_t))_t$, this property holds almost everywhere.
		Using the definitions we obtain
		\begin{align*}
			-\ddt f(y_t)  = 
			\limsup_{s \downarrow t} \frac{ f(y_t) - f(y_s) }{|t-s|}
			\leq
			\limsup_{s \downarrow t}
			\frac{f(y_t) - f(y_s)}{\mssd(y_t, y_s)} 
			\cdot 
			\limsup_{s \downarrow t}
			\frac{\mssd(y_t, y_s)}{|t-s|}
			\leq \slop{f}{y_t} \cdot \abs{\dot y_t}.
		\end{align*}
		Combining this inequality with \eqref{eq:EDI}, we find that
		\begin{align*}
			\tfrac{1}{2} \abs{\dot y_t}^2 
			+ \tfrac{1}{2}\slop{f}{y_t}^2
			\leq -\ddt f(y_t) 
			\leq
			\slop{f}{y_t} \cdot \abs{\dot y_t},
		\end{align*}
		which, again by Young's inequality, implies the desired identities.
	\end{proof}
	
	In light of Lemma~\ref{l:EqualityEDI}, every curve of maximal slope satisifies the \emph{Energy Dissipation Equality} for a.e.~ $t\in J$:
	\begin{equation}\tag{\mathsc{ede}}\label{eq:EDE}
		-\ddt f(y_t) 
		= 
		 \tfrac{1}{2} \abs{\dot y_t}^2 
		+\tfrac{1}{2}\slop{f}{y_t}^2.
	\end{equation}

	\begin{remark}\label{rmk: global AC2}
		Let $J =[a,b)$ with $-\infty < a < b \leq \infty$ and suppose that $y_a \in \dom{f}$.
		If $f$ is bounded from below by some constant $M\in \R$, then 
		$t \mapsto \abs{\dot y_t}$ belongs to 
		$L^2(a,b)$ 
		for every curve of maximal slope $(y_t)_{t \in J}$. 
		Indeed, for $t \in (a,b)$, integration of \eqref{eq:EDI} yields
		\begin{equation}\label{eq: energy identity}
			\frac{1}{2}\int_a^t \abs{\dot y_r}^2 \diff r +\frac{1}{2} \int_a^t \slo{f}{(y_r)}^2 \diff r \leq f(y_a) - f(y_t)
			\leq f(y_a) - M.
		\end{equation}
		The conclusion follows by passing to the limit
		$t \uparrow b$. 
	\end{remark}

	\begin{remark}[Comparison with~{\cite[Dfn.s~2.12, 2.13]{Hauer-Mazon:2019}}]
		Our Definition~\ref{d:CMS} is more restrictive than~\cite[Dfn.s~2.12, 2.13]{Hauer-Mazon:2019} as we additionally require the condition in~\ref{i:d:CMS:2}.
		This condition guarantees that~$\slo{f}$ is a strong upper gradient of~$f$ along~$\seq{y_t}_{t\in J}$; see e.g.~\cite[Rmk.~2.8]{AmbGigSav14}. 
		Furthermore, by~\ref{i:d:CMS:2} we may integrate~\eqref{eq:EDI} to conclude that~$t\mapsto f(y_t)$ is non-increasing, which is rather an assumption in~\cite[Dfn.~2.12]{Hauer-Mazon:2019}.
		Everywhere below, following~\cite{Hauer-Mazon:2019}, we could drop the assumption of~\ref{i:d:CMS:2} and replace~$\slo{f}$ by any given strong upper gradient~$g$.
		For the sake of simplicity however, we confine our exposition to the case~$g\eqdef\slo{f}$ for which the assumptions in~\cite{Hauer-Mazon:2019} are verified in light of~\ref{i:d:CMS:2} as discussed above.
	\end{remark}
	
	\section{Convergence of gradient flows}
	\label{sec:ConvGF}
	This section is devoted to the proof of Theorem \ref{thm:ConvRateIntro}, which deals with the converence of gradient flows under Assumption {$(A)$}.

	\begin{definition}[Equilibrium point]
		We say that~
		$x^\star\in X$ 
		is an \emph{equilibrium point} for~$f$ 
		if
		$x^\star\in\dom{\slo{f}}$ and~
		$\slo{f}(x^\star)=0$. 
	\end{definition}
	
	We refer to cf.~\cite[Dfn.~2.35]{Hauer-Mazon:2019} for a more general definition for strong upper gradients.

	Clearly, every local minimiser~$x^\star \in \dom{f}$ 
	is an equilibrium point for~$f$.

	It will be useful to first investigate gradient flow curves starting from an equilibrium point.
	
	\begin{lemma}[Trivial flows]\label{l:TrivialFlows}
		Let~$f\colon X\to (-\infty,\infty]$ be proper 
		and $T\in (0,\infty]$.
		\begin{enumerate}[(i)]
			\item If $x^\star\in \dom{f}$ 
			is an equilibrium point for~$f$, 
			then the constant curve~$\seq{y_t}_{t\in [0,T)}$ defined by~$y_t \equiv x^\star$ 
			is a curve of maximal slope for~$f$ starting at~$x^\star$. 
			\item If $x^\star\in \dom{f}$ is a local minimiser for~$f$, then 
			the constant curve~$\seq{y_t}_{t\in [0,T)}$ defined by~$y_t \equiv x^\star$ 
			is the only curve of maximal slope for~$f$ starting at~$x^\star$. 
		\end{enumerate}
	\end{lemma}
	
	\begin{proof}
		$(i)$: \ This follows immediately from the definitions.
		
		$(ii)$: \ 
		Let~$x^\star
		\in
		\dom{f}$ be a local minimiser and let~$U$ be a neighbourhood of~$x^\star$ such that
		 ~$f\geq f(x^\star)$ on~$U$. 	
		Furthermore, let $\seq{y_t}_{t\in [0,T)}$ be a curve of maximal slope for~$f$ starting at~$x^\star$, and set
		\[
		t_0\eqdef \inf\set{t>0: y_t\notin U}\wedge T \fstop
		\]
		Note that $t_0 > 0$, since $t \mapsto y_t$ is continuous. 
		Since $y_t \in U$ for $t \in [0, t_0)$, we have 
			$f(y_t) \geq f(x^\star)$ for $t \in [0, t_0)$. 
		As~$t\mapsto f(y_t)$ is non-increasing by ~\eqref{eq:EDI}, 
		we thus infer that~$f(y_t) = f(x^\star)$ for $t \in [0 ,t_0)$.
		Therefore,~$t\mapsto \ddt f(y_t)$ is identically~$0$, hence~$\abs{\dot y_t}=0$ for~$t \in [0,t_0)$ again by~\eqref{eq:EDI}.
		Applying \eqref{eq:MSpeed} to the metric speed, we infer that~
			$\mssd(y_t,y_0)
				\leq 
				\int_0^t \abs{\dot y_r} \diff r
				=
				0$, 
		hence~$y_t = y_0 \eqdef x^\star$ for all~$t\in [0,t_0)$.
		By continuity of~$t\mapsto y_t$ we conclude that
		$t_0=T$, which proves the assertion.
	\end{proof}

For convenience of the reader we recall the following definition from the introduction.

	\begin{definition}[Auxiliary  function]\label{d:Auxiliary}
		Given a parameter function $\theta : [0,\infty) \to [0,\infty)$ we consider the auxiliary function 
		\begin{align*}
			& \eta: [0,\infty)  \to [-\infty,\infty), 
			\qquad &
			\eta(u)
			& :=
			\int_1^u
			\big(\theta'(s)\big)^2
			\diff s
			\qquad 
			\text{for $u \in [0,\infty)$},\\
			&\Gamma : [0,\theta(\infty))  \to [-\infty,\infty), 
		\qquad	&
		\Gamma(u) & := (\eta \circ \theta^{-1})(u).
		\end{align*}
	\end{definition}
	Here we use the convention that $\theta(\infty) := \lim_{u \to \infty} \theta(u)$.
	Note that $\theta$ is indeed invertible and nonnegative, so that $\Gamma$ is well-defined.
	The following lemma collects some elementary properties of $\theta$. We leave the proof to the reader.
	
	\begin{lemma}[Properties of the auxiliary function]
		\label{lem:aux}
		The function $\eta$ is strictly increasing,   
		$\eta(1) = 0$, and
		$\eta(0)$ is possibly $-\infty$.
		Moreover, $\eta$ is continuously differentiable on $(0,\infty)$ and $\eta'(u) = (\theta'(u))^2$ for all $u > 0$.
	\end{lemma}

	\begin{remark}
		In the special case 
		where $\theta(u) = \frac{c}{\gamma} u^\gamma$ 
		we have the explicit formulas
		\begin{align*}
			\eta(u) = 
			\frac{c^2}{2\gamma -1} (u^{2\gamma -1} - 1)
			\quad\text{if $\gamma > 0,  \ \gamma \neq \frac12$},
			\qquad\text{and}\qquad
			\eta(u) = c^2 \log u 
			\quad\text{if $\gamma = \frac12$}
			\fstop
		\end{align*}
	\end{remark}

	The following lemma contains the crucial quantitative bounds on the distance and the driving functional that can be derived from Condition~{$(A)$}, for suitable gradient-flow trajectories that stay within the ball $B_r(x_0)$.

	\begin{lemma}[Distance bound and energy bound]\label{l:ThetaEstimate}
		Let~$f\colon X\to [0,\infty]$ be lower semicontinuous, and suppose that 
		$x_0\in\dom{f}$ and~$r>0$ 
		satisfy Condition~{$(A)$} 
		for some parameter function $\theta$.
		Let $\seq{y_t}_{t\in [0,T)}$, with~$T\in (0,\infty]$, be a curve of maximal slope starting at $x_0$. 
		Let $0 \leq s \leq t < T$ and
		assume that 
		$y_u \in B_r(x_0)$ 
		and
		$f(y_u) > 0$ for all $u \in [s,t]$.
		Then:
		\begin{align}
			\label{eq:l:ThetaEstimate:0}
			(\theta\circ f)(y_s)- (\theta\circ f)(y_t)
			& \geq 
			\mssd(y_t,y_s),\\
			\label{eq:l:EnergyEstimate:0}
			(\eta\circ f)(y_s) 
			- (\eta \circ f)(y_t)
			& \geq 
			t-s. 
		\end{align}		 
	\end{lemma}
	
	\begin{proof}
		As~$\theta$ and $\eta$ are continuously differentiable on~$(0,\infty)$, 
		and ~$t\mapsto f(y_t)$ is 
		locally absolute continuous,
		we conclude that also ~$\mcH\colon t\mapsto (\theta\circ f)(y_t)$
		and 
			$t\mapsto (\eta \circ f)(y_t)$
		are locally absolutely continuous on~$(0,T)$.
		For almost every $u \in [s,t]$, we obtain
		by absolute continuity of~$\mcH$, by~\eqref{eq:EqYoung}, and by~{$(A)$},
		\begin{equation}\label{eq:l:ThetaEstimate:2}
			-\mcH'(u)
			= -(\theta'\circ f)(y_u) \cdot \ddu f(y_u) 
			= (\theta'\circ f)(y_u) \cdot \slo{f}(y_u) \abs{\dot y_u} 
			\geq \abs{\dot y_u}.
		\end{equation}
		Since~$t \mapsto y_t$ is locally absolute continuous, we obtain
		\begin{align}\label{eq:l:ThetaEstimate:3}
			\mssd(y_t,y_s) 
			\leq 
			\int_s^t \abs{\dot y_u} \diff u
			\leq \int_s^t - \mcH'(u) \diff u 
			= \mcH(s)-\mcH(t),
		\end{align}
		which proves \eqref{eq:l:ThetaEstimate:0}.
		
		Moreover, 
		using again that $f(y_u) > 0$ for all $u \in [s,t]$, we obtain for a.e. $u \in (s,t)$,
		by Lemma \ref{lem:aux}, 
		by Lemma \ref{l:EqualityEDI},
		and by~{$(A)$},
		\begin{align*}
			- \ddu (\eta\circ f)(y_u)
			=
			- (\eta'\circ f)(y_u) 
			\cdot
			\ddu f(y_u)
			= 
			\Big(
			(\theta'\circ f)(y_u) \cdot \slo{f}(y_u)	
			\Big)^2
			\geq 1.			
		\end{align*}
		Integration of this inequality yields \eqref{eq:l:EnergyEstimate:0}.
	\end{proof}
	
	\begin{ese}\label{ese:explicit-f-decay}
		An explicit computation shows that in the special case where $\theta(u) = \frac{c}{\gamma} u^\gamma$, the energy estimate~\eqref{eq:l:EnergyEstimate:0} becomes
		\begin{align}\label{eq:ex:en-est-1}
			f(y_t) &\leq 
			\begin{cases}
				\displaystyle
				\bigg(
			f(y_s)^{2\gamma -1}
			- \frac{2\gamma -1}{c^2}(t-s)
			\bigg)^{1/(2\gamma -1)}
			& 
			\text{ if }
			\gamma > 0\comma  \gamma \neq \frac12
			\comma
			\\
			e^{-(t-s)/c^2}f(y_s) & 
			\text{ if } 
			\gamma = \frac12 \fstop	
			\end{cases} 
		\end{align}
	\end{ese}

	We are now ready to prove our first main result.
	\begin{proof}[Proof of Theorem \ref{thm:ConvRateIntro}]
		We assume that $f(x_0)>0$, as the result would otherwise follow immediately from Lemma~\ref{l:TrivialFlows}.
		
		\ref{i:thm-conv-1} \
		We define
		\begin{align*}
			t_0
			\eqdef  
			\inf \big\{t\in [0,T): 
			y_t\in \partial B_r(x_0)\big\} \wedge T
		\end{align*}
		and note that $t_0 > 0$, since $(y_t)_{t \in[0,T)}$ is continuous.
		If $t_0 = T$ the conclusion follows,
		hence it suffices to treat the case where $t_0 < T$.

		If $f(y_{t_0}) = 0$, the conclusion follows from Lemma \ref{l:TrivialFlows} and the definition of $t_0$. 
		It thus remains to treat the case where  
		$t_0 < T$ and $f(y_{t_0}) > 0$. 
		We will show that these conditions yield a contradiction, which completes the proof.
		
		Indeed, 
		\eqref{eq:l:ThetaEstimate:0} 
		and
		Assumption {$(A)$}
		yield,
		for $0 < t < t_0$,
		\begin{align*}
			\mssd(y_t,x_0)
			& \leq
			(\theta\circ f)(x_0) - (\theta\circ f)(y_t)
			\leq  r - (\theta\circ f)(y_{t_0}) \fstop
		\end{align*}
		Since $(\theta\circ f)(y_{t_0}) > 0$ 
		and  $t \mapsto y_t$ is continuous, 
		it follows by passing to the limit
		$t \uparrow t_0$ that 
		$\mssd(y_{t_0},x_0) < r$.
		This is the desired contradiction, since 
		$\mssd(y_{t_0},x_0) = r$ by construction.

		\smallskip
		\ref{i:thm-conv-2} \ 
		Since $t \mapsto f(y_t)$ is continuous, it follows that 
		\[
		t_*\eqdef \inf\set{t\in [0,T): f(y_t)=0} \wedge T >0\fstop
		\]		
		We first claim that $y_T\eqdef \lim_{t\to T} y_t$ exists and belongs to $\overline{B_r(x_0)}$. 

		If~$t_*<T$, then~$y_t=y_{t_*}$ for every~$t\in [t_*, T)$ by Lemma~\ref{l:TrivialFlows}, 
		and the claim follows.

		If otherwise~$t_*=T$, then~\eqref{eq:l:ThetaEstimate:0} holds for all 
		$0\leq s\leq t < T$. Write $\mcH(t) \eqdef (\theta \circ f) (y_t)$.
		Then $\mcH\colon [0,T)\to [0,\infty)$ is continuous, non-increasing and bounded from below, so it admits a continuous non-increasing extension on $[0,T]$.
		Thus, the bound (for $0\leq s<t<T$) 
		\[
		\mssd(y_s,y_t) \leq \mcH(s)-\mcH(t) \leq \mcH(s) - \mcH(T)
		\]
		combined with $\mcH(s) \downarrow \mcH(T) \geq 0$ as $s\to T$ implies the Cauchy property of $\seq{y_t}_t$, hence the existence of the limit, which proves the claim. 
		
		By lower semicontinuity of $f$ and Lemma \ref{l:TrivialFlows} and in view of \ref{i:thm-conv-1}, we infer that \eqref{eq:l:ThetaEstimate:0} holds for all $0\leq s\leq t \leq T$ (even if $t_*<T$). 
		Choosing 
		$s= 0$ and $t= T$, 
		the last part of the statement follows using
		\ref{eq:ass:cts}.

		\smallskip
		
		\ref{i:thm-conv-3} \
		Let $0\leq t<t_*$.
		In view of \ref{i:thm-conv-1}, \eqref{eq:f-rate} follows from \eqref{eq:l:EnergyEstimate:0}.
		Next, by \eqref{eq:theta-estimate-thm} we have 
		\[
		\mssd(y_t,y_T) \leq \theta(f(y_t)) - \theta(f(y_T)) \leq \theta(f(y_t)).
		\]
		Using this bound and \eqref{eq:f-rate}, we obtain
		\begin{align*}
			(\eta \circ \theta^{-1}) (\mssd(y_t, y_T))
			\leq 
			(\eta \circ f)(y_t)
			\leq
			(\eta\circ f)(x_0) 
			- t
			\leq
			(\eta\circ \theta^{-1})(r) 
			- t,
		\end{align*}
		which shows \eqref{eq:d-rate}.
		By continuity of $(y_t)_t$ and lower semicontinuity of $f$, \eqref{eq:d-rate} and \eqref{eq:f-rate} extend to $t=t_*$.
		
		Finally, suppose that $T=\infty$. If $t_*<\infty$, then clearly $f(y_\infty) = 0$. If on the other hand $t_* = \infty$, 
		it follows from \eqref{eq:f-rate} 
		that
		$(\eta \circ f)(y_t) \to -\infty$
		as $t \to \infty$, hence 
		$f(y_t) \to 0$.
		By lower semicontinuity of $f$ the result follows.
	\end{proof}
	
	In the special case where the parameter function $\theta $ takes the form $\theta(u) = \frac{c}{\gamma}u^\gamma$, we obtain the following more explicit result.
	The notation $t^*$ was introduced in Theorem \ref{thm:ConvRateIntro}.

	\begin{corollary}
		\label{cor:conv-rate-gf-special}
		Let~$f\colon X\to [0,\infty]$ be lower semicontinuous and suppose that
		$x_0\in \dom{f}$ and~$r>0$
		satisfy Condition~{$(A)$}
		with parameter function~$\theta(u)=\tfrac{c}{\gamma}u^\gamma$ for some~$c>0$ and~$\gamma \in (0,1]$.
		Let~$\seq{y_t}_{t\in [0,T)}$ be a curve of maximal slope for~$f$ starting at~$x_0$,
		for some $T \in (0,\infty]$. 
		Then 
		$y_T\coloneqq \lim_{t\to T} y_t$ exists, 
		$y_t$ belongs to $\overline{B_r(x_0)}$ 
		for all  $t \in [0,T]$,
		and
		for all $0\leq t\leq t^*$
		we have
		\begin{align}
		\label{eq:c:rate-gf-special:00}
			\mssd(y_t,y_T) &\leq 
			\frac{c}{\gamma}
				\bigg(
					\Big(\frac{\gamma r}{c}\Big)
				^{\frac{2\gamma-1}{\gamma}}
				- 
				\frac{2\gamma-1}{c^2}t\bigg)^{\frac{\gamma}{2\gamma -1}},
			&
			f(y_t) 
			& \leq \bigg(
				f(x_0)^{2\gamma -1}
				- \frac{2\gamma -1}{c^2}t
				\bigg)^{\frac{1}{2\gamma -1}},
			& 
			\text{if } \gamma &\neq \tfrac12 \comma
			\\
			\label{eq:c:rate-gf-special:0}
			\mssd(y_t,y_T) &\leq r e^{-\frac{t}{2c^2}},
			&
			f(y_t) &\leq f(x_0)\, e^{-\frac{t}{c^2}},
			& 
			\text{if } \gamma & = \tfrac12\fstop
		\end{align}
		Moreover, 
		if $\frac12 < \gamma \leq 1$, we have
		$t_* \leq \frac{c^2}{2\gamma - 1} f(x_0)^{2\gamma -1}$.
		\end{corollary}

	\begin{proof}
		The estimates for $\mssd(y_t,y_T)$ and $f(y_t)$
		are obtained from
		\eqref{eq:d-rate} and \eqref{eq:f-rate}
		by rearranging terms.
		The final assertion follows from the second bound in~\eqref{eq:c:rate-gf-special:00}.
	\end{proof}

		\begin{remark}[The case $\theta(u)=\tfrac{c}{2}\sqrt u$]
		\label{}
		Under the above assumptions,
		the distance estimate in Corollary \ref{cor:conv-rate-gf-special} 
		can be improved
 		if $\gamma =\frac{1}{2}$, 
		using the following the ideas of \cite{Cha22}.
		Let $0 \leq s \leq t \leq T$ and
		assume that 
		$y_u \in B_r(x_0)$ 
		and
		$f(y_u) > 0$ for all $u \in [s,t]$.
		Then:
		\begin{align}
		\label{eq:t:Main:3-old}
		\mssd(y_t,y_s)^2 
		& \leq 
	4c^2
		\Big(
		e^{-\frac{s}{2c^2}}-e^{-\frac{t}{2c^2}}
		\Big)
		\sqrt{f(x_0)}
		\Big(\sqrt{f(y_s)} - \sqrt{f(y_t)}\Big)
		\\& \leq
		4c^2
		e^{-\frac{s}{2c^2}} 
		\Big(e^{-\frac{s}{2c^2}}-e^{-\frac{t}{2c^2}}\Big) 
		f(x_0) \fstop
		\label{eq:t:Main:3.5}
		\end{align}		 
	\end{remark}
	
	\begin{proof}
		We can assume $t<T$ and then extend the result to $t=T$ by taking limits. Using the local 2-absolute continuity of
		$\seq{y_t}_{t\in [0,T)}$, 
		the Cauchy--Schwarz inequality,
		and the assumption that $f(y_u) > 0$ for all $u \in [s,t]$, we find
		\begin{align}
		\label{eq:d-bound}
		\mssd(y_t,y_s) 
		\leq 
		\int_s^t \abs{\dot y_u} \diff u
		\leq 
		\paren{\int_s^t \sqrt{f(y_u)} \diff u}^{1/2} 
		\paren{\int_s^t 
			\frac{\abs{\dot y_u}^2}{\sqrt{f(y_u)}}\diff u}^{1/2} \fstop
		\end{align}
		By local absolute continuity of~$u\mapsto f(y_u)$ we see that~$u \mapsto \sqrt{f(y_u)}$ too is locally absolutely continuous, and 
		$\ddu \sqrt{f(y_u)}(t) 
		= \big(2\sqrt{f(y_u)}\big)^{-1}\ddu f(y_u)$ holds 
		a.e.\ on~$(0,T)$.
		Since $\ddu f(y_u) = - \abs{\dot y_u}^2$		
		by ~\eqref{eq:EqYoung}, we obtain
		\begin{align}
		\label{eq:speed}
		\int_s^t \frac{\abs{\dot y_u}^2}{\sqrt{f(y_u)}} \diff u 
		=  -\int_s^t \frac{\ddu f(y_u)}{\sqrt{f(y_u)}} \diff u
		= -2 \int_s^t \ddu \sqrt{f(y_u)} \diff u
		= 2\Big(\sqrt{f(y_s)}-\sqrt{f(y_t)}\Big) \fstop
		\end{align}
		Since $\seq{y_u}_{u \in[s,t]} \subseteq B_r(x_0)$, we can take the square root of the second bound in \eqref{eq:c:rate-gf-special:0} 
		to see that
		\begin{align}
		\label{eq:sqrt}
		\int_s^t  \sqrt{f(y_u)}\diff u
		\leq 
		\sqrt{f(x_0)} 
		\int_s^t  e^{-\frac{u}{2c^2}}
		\diff u 
		= 
		{2}c^2
		\tond*{e^{-\frac{s}{2c^2} }-e^{-\frac{t}{2c^2}}}
		\sqrt{f(x_0)}\fstop
		\end{align}
		Inserting \eqref{eq:speed} and \eqref{eq:sqrt} into \eqref{eq:d-bound}, we arrive at \eqref{eq:t:Main:3-old}.
		
		Finally, another application of the second bound in \eqref{eq:c:rate-gf-special:0} yields
		\begin{align*}
		\sqrt{f(y_s)} - \sqrt{f(y_t)}
		\leq \sqrt{f(y_s)}
		\leq \sqrt{f(x_0)}\, e^{-\frac{s}{2c^2}}.
		\end{align*}
		Inserting this inequality into the right-hand side of \eqref{eq:t:Main:3-old} we obtain \eqref{eq:t:Main:3.5}.
	\end{proof}

	\section{Comments on the assumption}
	\label{sec:assumption}
	
	In this section we collect some comments on the main assumption of this paper,
	Conditions~{$(A)$} and~{$(A')$} 
	introduced in Definition \ref{def:main-ass}.

	\begin{remark}[Comparison with~\cite{Cha22}]\label{rmk:comp:cha}
		Let~$f\colon X\to [0,\infty]$ be proper.
		For~$r>0$ and~$x_0\in \dom{f}$ with $f(x_0)>0$ we define
		\begin{align}\label{eq:def-alpha-comp}
			\alpha = \alpha(x_0, r) 
			\eqdef 
			\inf_{\substack{x\in B_r(x_0)\\0<f(x)\leq f(x_0)} }
			\frac{\slop{f}{x}^2}{f(x)}\fstop
		\end{align}
		If $0< \alpha < \infty$,
		it follows immediately from the definitions that 
		the following statements are equivalent:
		\begin{enumerate}[$(i)$]
			\item\label{i:r:ChatterjeeComparison:1}  Condition~{$(A)$} holds for the parameter function~$\theta(u)\eqdef 2\sqrt{u/\alpha(x_0,r)}$;
			\item\label{i:r:ChatterjeeComparison:2} 
			The following inequality holds:
			\begin{align}\label{eq:r:ChatterjeeComparison:0}\tag{$C$}
				\alpha(x_0,r)
				\geq 
				\frac{4 f(x_0) }{r^2} \fstop
			\end{align}
		\end{enumerate}
		
		Similarly, the slightly stronger Condition {$(A')$} from 
		Definition \ref{def:main-ass} is equivalent to Condition {$(C')$},
		the strict inequality 
		$\alpha(x_0,r)
		>
		\frac{4 f(x_0) }{r^2}$.
		The latter condition is essentially identical to the main standing assumption in \cite{Cha22},
		in the setting of $C^2$ functions on $\R^n$.
		The difference is that we restrict 
		the infimum in~\eqref{eq:def-alpha-comp}
		to a sub-level set of $f$ and work with an open ball instead of a closed ball of radius $r$ around $x_0$.
		
	\end{remark}
	
	The following example illustrates that it is occasionally useful to consider the weaker Condition~$(C)$ instead of Condition~$(C')$.

	\begin{ese}
		Fix $x_0>0$ and consider the function $f\colon\R\to\R$ defined by (see Fig.~\ref{fig:StrictIneq})
		\begin{equation}\label{eq:f-StrictIneq}
			f(x) = 
			\begin{cases}
				x^2 & \text{ if } x\geq 0 \\
				\frac{x_0^2}{2} & \text{ if } x<0
			\end{cases} \fstop
		\end{equation}
		\begin{figure}[htb!]
			\includegraphics[scale=.75,trim=100 0 0 150, clip]{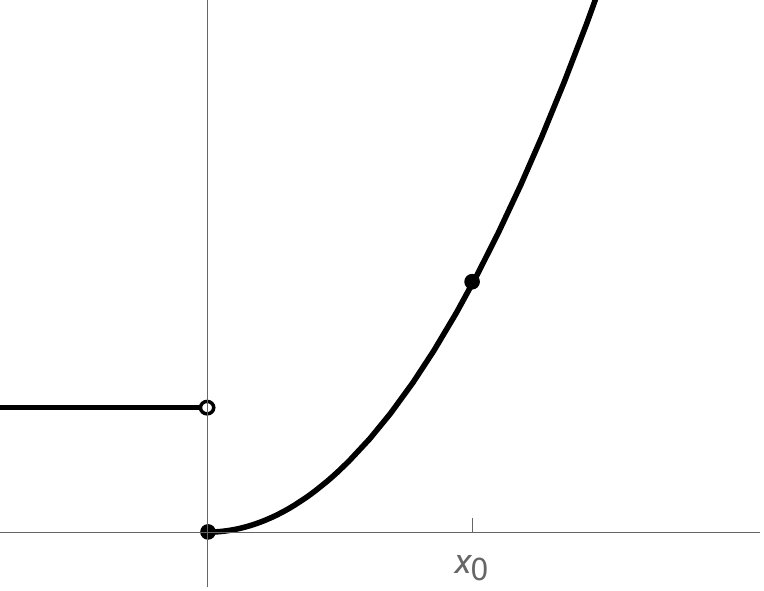}
			\caption{The function in~\eqref{eq:f-StrictIneq}.}
			\label{fig:StrictIneq}
		\end{figure}
		Then $\alpha(x_0, r) = 4$ for $0 < r \leq x_0$ and $\alpha(x_0, r) = 0$ for $r > x_0$.
		Therefore, Condition~$(C')$ fails to hold regardless of the choice of $r > 0$,
		but Condition~$(C)$ is satisfied for $r = x_0$.
	\end{ese}

\begin{remark}[Attainment of the minimum]\label{rmk:IoffeBasicLemAppl}
	Assumption~$(A')$
	implies
	the existence of a global minimiser~$x^\star$
	of 
	$f$
	satisfying 
	$\mssd(x^\star ,x_0) \leq 
	(\theta\circ f)(x_0)$
	and
	$f(x^\star) = 0$.
	This
	follows from 
	a result by Ioffe \cite{Iof77}, 
	which we recall in 
	Lemma~\ref{lem:Ioffe-lemma} below.
	To 
	derive the conclusion,
	Ioffe's result
	should be to be applied to 
	the function~$\theta \circ {f}$,
	and  
	a metric version of the chain rule is required to relate the slope of $f$ to the slope of~$\theta\circ {f}$.
	For completeness, we give a proof of this chain rule in Lemma \ref{l: chain rule for slope}.
	
	In light of this observation,
	it is possible to derive  
	results similar to Theorem \ref{thm:ConvRateIntro}
	by applying existing results 
	for convergence to a global minimum 
	under the K{\L} condition that assume the existence of a global minimum close to the starting point $x_0$;
	see, e.g., \cite[Thm.~10]{AttBolRedSou2010} and \cite[Thm.~2.12]{Att-Bol-Sva-2013} for such results for discrete schemes.
	However, a combination of these results with Ioffe's result yields a non-optimal criterion, as the K{\L} inequality is required to hold on a bigger set than necessary. 
	Moreover, some additional assumptions are made in the aforementioned results.
\end{remark}

	\begin{remark}[Sharpness of Condition {$(A)$}]
		To guarantee the existence and the proximity of a  global minimiser of $f$ under Condition {$(A)$}, the constant $r$ in the inequality $(\theta\circ f)(x_0) \leq r$ cannot be replaced by any larger constant. 

		To see this, fix $M < \infty$ (large) and consider for (small) $\eps \geq 0$ the function 
			$f_\eps : [0,\infty) \to [0,\infty)$ defined by
			$f_\eps(x) = \theta^{-1}(x + \eps)$ for $0 \leq x <M$
			and $f_\eps(x) = 0$ for $x \geq M$.
		Fix $x_0 \in (0, M/2)$.
		Differentiating the identity
			$\theta(f_\eps(x)) = x + \eps$, 
			we find that 
			$(\theta' \circ f_\eps)(x) f'_\eps(x)  
			=  1$
			for $0 < x  < M$.
			In particular, 
			the second inequality in Condition $(A)$ is  satisfied in an open ball of radius $x_0$ around $x_0$.

		If $\eps = 0$, 
		the identity $\theta(f_0(x_0)) = x_0$ implies that 
		Condition $(A)$ holds with $r = x_0$, 
		and indeed, the distance of $x_0$ to the nearest global minimiser of $f$ (which is $0$) equals 
		$x_0$.
		
		If $\eps > 0$, Condition $(A)$ fails to hold just barely (since $\theta(f_\eps(x_0)) = x_0 + \eps$), but the distance of $x_0$ to the nearest global minimiser (which is $M$) is enormous (namely, $M - x_0$)
		and the gradient flow curve starting from $x_0$ will converge to $0$, which is not a global minimiser.
	\end{remark}
		
		The following non-smooth example in $\R$ shows that Condition~{$(A)$} 
		can be applied in a setting where there is no uniqueness of gradient flow curves with a given starting point.

	\begin{ese}[Non-uniqueness]
		\label{ex:non-uniqueness}
		Let $\lambda > 0$ and $a > 0$, and consider the function $f \colon \R \to \R$ (see Fig.~\ref{fig:NonUniqueness:1})
		\begin{align}
			\label{eq:f-non-unique}
			f(x) = \min 
			\bigg\{ 
			\frac{\lambda}{2} (x-a)^2, \
			\frac{\lambda}{2} (x+a)^2
			\bigg\} \fstop
		\end{align}
		This function is everywhere smooth except at the origin.
		For each $x_0 \neq 0$, there exists a unique gradient-flow trajectory starting at $x_0$, given by 
		$y_t \eqdef 
		e^{-\lambda t} x_0 
		\pm \big( 1 - e^{-\lambda t} \big)a$
		for $x_0 \gtrless 0$.
		However, there are two distinct gradient-flow trajectories 
		$y^+$ and $y^-$ 
		starting at the origin, given
		$y^\pm_t \eqdef \pm \big( 1 - e^{-\lambda t} \big)a$
		for $t \geq 0$.
		
		In spite of this non-uniqueness, 
		we shall verify that this example satisfies our assumptions.
		Note that 
		$\slop{f}{x}
		= 
		\lambda  \tabs{|x|-a}$
		for all $x \in \R$.
		In particular, $f$ has finite slope at $0$, although it is not differentiable.
		Consequently, 
		$\frac{\slop{f}{x}^2}{f(x)} = 2 \lambda$
		for all $x \in \R$.
		It follows that Condition~\eqref{eq:r:ChatterjeeComparison:0} holds for all $x_0 \in \R$ with 
		$\alpha(x_0, r) = 2 \lambda$ (hence Condition~$(A)$ holds with~$\theta(u)=\sqrt{2u/\lambda}$), 
		provided 
		$r \geq | x_0 - a| \wedge | x_0 + a |$.
		Thus, at every point $x_0 \in \R$, the criterion provides the optimal result, in the sense that it yields the smallest possible ball centered at $x_0$ containing each gradient-flow trajectory starting at $x_0$.
	\end{ese}

	\begin{remark}[Restriction to path connected component]\label{r:OtherAlpha}
		The second inequality in Condition~{$(A)$} is required to hold for all 
		$y \in B_r(x_0)\cap \{0<f \leq f(x_0)\}$.
		However, in the proof of Theorem \ref{thm:ConvRateIntro}, this bound is needed only on the set $G(x_0,r)$ consisting of all points inside the ball that are reachable by the considered curve of maximal slope starting at $x_0$. 
		Therefore, Theorem \ref{thm:ConvRateIntro} would still hold if one replaces the set
		$B_r(x_0) \cap 
		\{ 0<f \leq f(x_0)\}$
		by $G(x_0, r)$ 
		in the definition in {$(A)$}.
		Of course, in practice $G(x_0,r)$ is often not explicitly known, so this condition might be not easy to check. 
		Instead of $G(x_0,r)$, one could also consider the path connected component $P(x_0,r)$ of $x_0$ in $B_r(x_0)\cap \{ 0<f \leq f(x_0)\}$ 
		and modify the definition of~{$(A)$} accordingly.
	
		The following modification of Example~\ref{ex:non-uniqueness} provides an example where it is useful to employ the modified assumption.
		Let $\lambda > 0$ and $a > 0$, and consider the function~$f \colon \R \to \R$ (see Fig.~\ref{fig:NonUniqueness:2}) given by
		\begin{align}
			\label{eq:f-PathConnected}
			f(x) = \min 
			\set{ 
				\max\set{\frac{\lambda}{2} (x-a)^2,\epsilon} \
				\frac{\lambda}{2} (x+a)^2
			} \fstop
		\end{align}
		For this function, Assumption~\eqref{eq:r:ChatterjeeComparison:0} is satisfied for every~$x_0<0$ and suitable $r>0$ when~$\alpha(x_0,r)$ is defined with~$P(x_0,r)$ in place of~$B_r(x_0)\cap \set{f\leq f(x_0)}$.
		However, it is not satisfied for any~$x_0<0$ yet sufficiently close to~$0$ when~$\alpha(x_0,r)$ is defined as in~\eqref{eq:def-alpha-comp}.
\end{remark}

	\begin{figure}[htb!]
		\begin{subfigure}[c]{0.45\textwidth}
			\includegraphics[scale=.5]{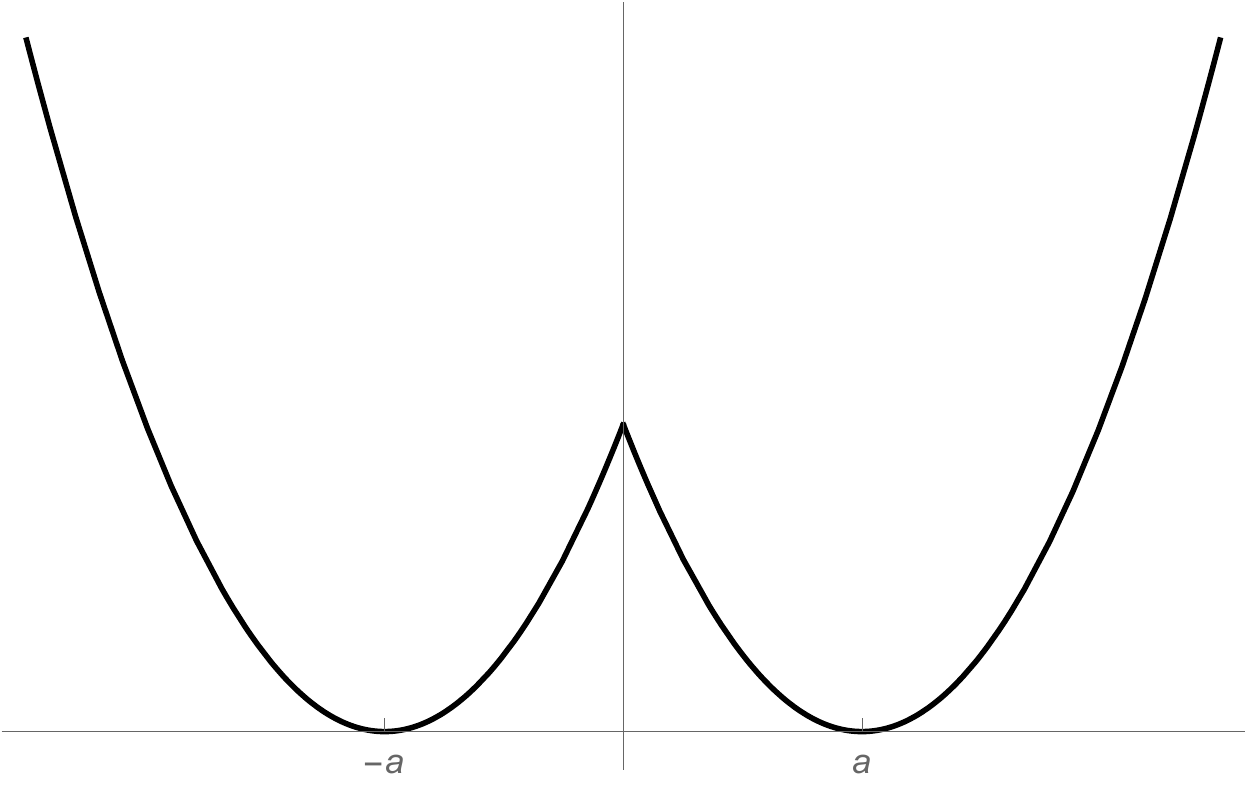}
			\subcaption{A non-smooth double-well potential~\eqref{eq:f-non-unique}.}
			\label{fig:NonUniqueness:1}
		\end{subfigure}
		\qquad
		\begin{subfigure}[c]{0.45\textwidth}
			\includegraphics[scale=.5]{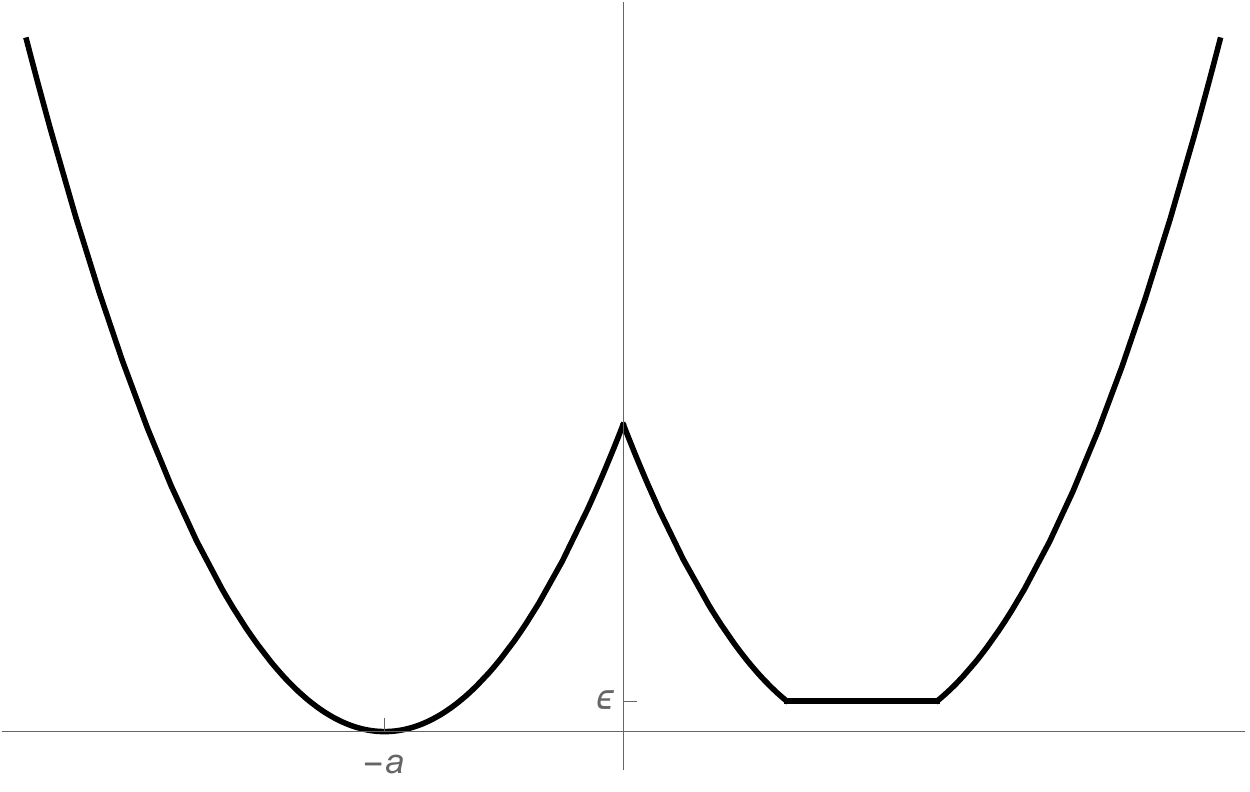}
			\subcaption{An asymmetric double-well potential~\eqref{eq:f-PathConnected}.}
			\label{fig:NonUniqueness:2}
		\end{subfigure}
		\caption{ The objective functions in Example~\ref{ex:non-uniqueness} and~ Remark \ref{r:OtherAlpha}}\label{fig:NonUniqueness}
	\end{figure}

	\section{Extension of gradient-flow trajectories}\label{sec:Extension}

	It is possible, even under Condition~$(A)$, that a curve of maximal slope defined on a finite interval~$[0,T)$ does not extend to a curve of maximal slope on~$[0,\infty)$.
	The following simple example illustrates this phenomenon.

	\begin{ese}
		\begin{figure}[htb!]
			\includegraphics[scale=.5]{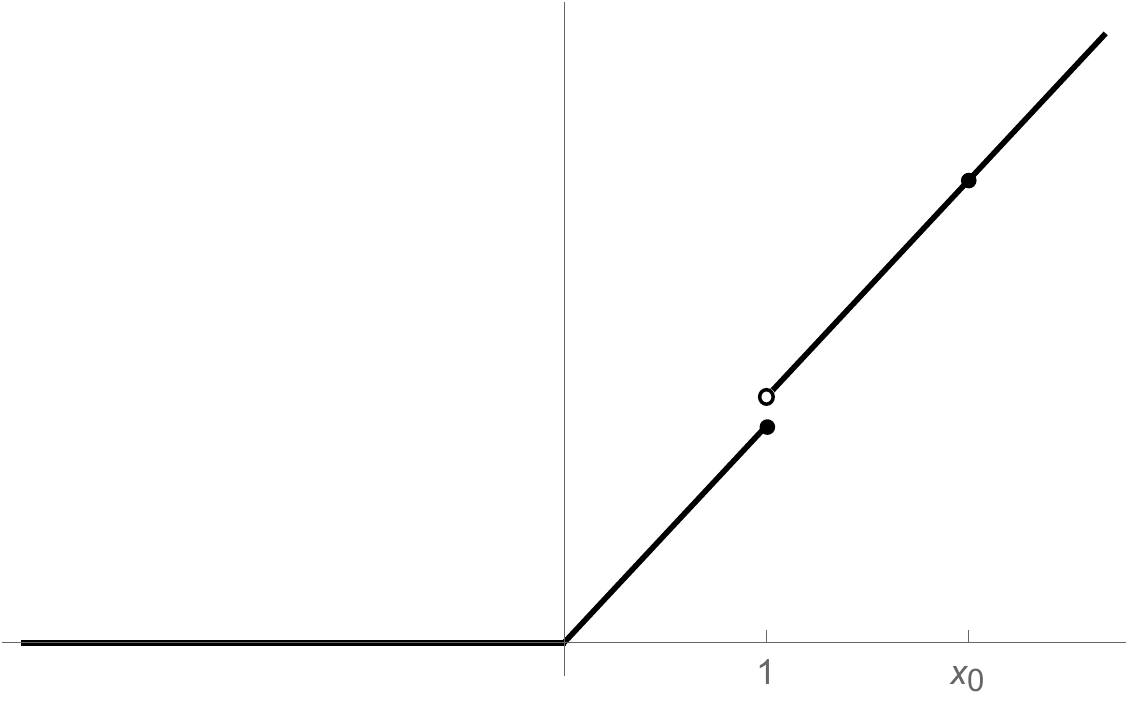}
			\caption{There is no curve of maximal slope with $T=\infty$ starting at~$x_0$.}
			\label{fig:discontinuity}
		\end{figure}
		
		For fixed~$m,\eps>0$, consider the lower-semicontinuous function $f\colon \R \to [0,\infty)$ defined by
		\[
		f(x) = 
		m\,  x \, {\bf 1}_{[0,1]}(x)
		+
		(m x + \eps) {\bf 1}_{(1,\infty)}(x) \fstop
		\]
		See Figure \ref{fig:discontinuity}.
		Let $x_0 >1$ and fix $r>0$. Then $f(x_0) =m x_0+\eps$ and
		Condition~$(A)$ is satisfied with 
		\[
		\theta(u) = \frac{2}{m} \sqrt{u (m x_0 + \epsilon)}
		\]
		and
		$
		r \geq 2(m x_0+\eps)/m.
		$
		On the interval $[0,T)$
		with $T = \frac{x_0-1}{m}$,
		there exists a unique curve of maximal slope 
		$\seq{y_t}_{t\in [0,T)}$ 
		starting from $x_0$.
		This is the curve which travels at constant speed $m$ towards the discontinuity of $f$, namely
		$y_t = x_0 - m t$.
		However, there is no extension of~$\seq{y_t}_{t\in [0,T)}$ to a curve of maximal slope defined on~$[0,T')$ for any~$T'>T$, since~$t\mapsto f(y_t)$ cannot be (absolutely) continuous on~$[0,T')$.
	\end{ese}
	
	Of course, the curve in this example can be naturally extended to $[0,\infty)$ by defining $y_T = 1$, and concatenating a new curve of maximal slope starting from there. 
	The resulting curve, given by
	$y_t = (x_0 - m t)_+$ for $t \geq 0$,
	satisfies the exponential convergence rates of Corollary \ref{cor:chatt-analogue-cts}, even though it is not a curve of maximal slope in the sense of Definition \ref{d:CMS}.
	
	Theorem \ref{t: convergence glued curves} below shows that, under Condition~$(A)$, concatenated curves of maximal slope always satisfy the convergence rates of Theorem \ref{thm:ConvRateIntro}.
	The key ingredient is the following simple observation, which shows that Condition~$(A)$ is preserved under curves of maximal slope $(y_t)_{t \geq 0}$ in a suitable sense: 
	if the condition holds at time $0$ for~$x_0$ and some $r>0$, then it holds at any time $t \geq 0$ for the point~$y_t$ and the radius 
	$r - \mssd(y_t, y_0)$ and with the same parameter function $\theta$.

	\begin{remark}[Assumption preserved along the flow]\label{rmk:ass-pres-flow}
		Suppose that Condition {$(A)$} holds for $x_0\in \dom{f}$ and $r>0$. Let $\seq{y_t}_{t\in [0,T)}$ be a curve of maximal slope, and extend it to $\seq{y_t}_{t\in[0,T]}$ using Theorem \ref{thm:ConvRateIntro}. 
		Then we have by~\eqref{eq:l:ThetaEstimate:0} that for $t\in [0,T)$,
		\[
		(\theta \circ f)(y_t) \leq (\theta \circ f)(x_0) - \mssd(x_0,y_t) \leq r - \mssd(x_0,y_t),
		\]
		which implies that Condition~$(A)$ holds for $y_t$ (in place of~$x_0$) and $r-\mssd(x_0,y_t)$ (in place of~$r$); note that if $f(y_t)=0$ it is possible that $r-\mssd(x_0,y_t) =0 $; otherwise this quantity is strictly positive. 
		If $f$ is lower semicontinuous, then 
		Condition $(A)$ holds also for~$y_T$ and~$r-\mssd(x_0,y_T)$,
		as can be seen by taking limits.
	\end{remark}

	\begin{theorem}\label{t: convergence glued curves} 
		Let $f : X \to [0, \infty]$ be a lower semicontinuous function on a complete metric space $(X,\mssd)$ 
		and suppose that 
		$x_0 \in X$ and $r > 0$ satisfy Condition~$(A)$ for some parameter function $\theta$.
		Let $K\geq 1$, $0 = T_0 < T_1 <\ldots < T_K = T \leq \infty$ and $(y_t)_{t\in [T_i, T_{i+1}) }$ be curves of maximal slope starting from $x_i$, where 	$x_i = \lim_{t \uparrow T_i} y_t \in B_r(x_0)$ for $1\leq i \leq K-1$. Then, setting $t_*\coloneqq \inf\set{t\in [0,T): f(y_t)=0} \wedge T$, the following assertions hold:
		\begin{enumerate}[$(i)$]
			\item \label{it:thm-glue confinement} $y_t \in \overline{B_r(x_0)}$ for all $0\leq t < T$;
			\item \label{it:thm-glue limit}  
			$y_T := \lim_{t\to T} y_t$ exists and belongs to $\overline{B_r(x_0)}$;
			\item \label{it:thm-glue decay distance}$\mssd(y_s,y_t) \leq (\theta\circ f)(y_s) - (\theta\circ f)(y_t)$ for all $0\leq t \leq T$.
			\item \label{it:thm-glue decay functional} 		for all  $0 \leq t \leq t_*$
			\begin{align}
				\label{eq:d-rate-glue}
				\Gamma\big(\mssd(y_t,y_T)\big)& \leq \Gamma(r) - t, 
				\\
				\label{eq:f-rate-glue}
				(\eta\circ f)(y_t)		
				& \leq 
				(\eta\circ f)(x_0)
				- t.
			\end{align} 
		\end{enumerate}
	\end{theorem}

	\begin{proof}

		\ref{it:thm-glue confinement} and  \ref{it:thm-glue limit} follow from a repeated application of Theorem \ref{thm:ConvRateIntro} and Remark \ref{rmk:ass-pres-flow}.

		\smallskip
		\ref{it:thm-glue decay distance}: \ 
		Recall first that, for $K = 0, \ldots, K - 1$ and 
		$T_k \leq s \leq t \leq T_{k+1}$, by \eqref{eq:theta-estimate-thm}
		\begin{align*}
			\mssd(y_s,y_t) \leq (\theta\circ f){(y_s)} -(\theta\circ f){(y_t)} \fstop
		\end{align*}
		Therefore by a telescoping sum argument, the same inequality holds for all $0 \leq s \leq t \leq T$.
		
		\smallskip

		\ref{it:thm-glue decay functional}: \
		If $K = 1$, the claim follows from Theorem \ref*{thm:ConvRateIntro}.
		Proceeding by induction, we assume that the claim holds for all $K \leq \bar K$. We shall show that it also holds for $K = \bar K + 1$.
		For this purpose, suppose that $t_*\geq \bar{K}$ and let $T_{\bar{K}}\leq t\leq t_*$, otherwise the conclusion is trivial. Then notice that the induction hypothesis yields
		$(\eta\circ f)(x_{\bar K}) \leq (\eta\circ f)(x_0) - T_{\bar{K}}$.
		Moreover, applying Theorem \ref{thm:ConvRateIntro} for $\tond{y_t}_{ t\in [T_{\bar K}, T_{\bar K + 1}]}$, we find
		\begin{align*}
			(\eta\circ f)(y_t) \leq (\eta\circ f)(x_{\bar K}) - (t-T_{\bar{K}}).
		\end{align*}
		Combining these bounds, \eqref{eq:f-rate-glue} follows.
		Finally, \eqref{eq:d-rate-glue} follows from \ref{it:thm-glue decay distance} and \eqref{eq:f-rate-glue} in the same way as in the proof of Theorem \ref{thm:ConvRateIntro}.
	\end{proof}

	\begin{remark}
		Theorem \ref{t: convergence glued curves} still holds true if we replace the sub-level set of~$f$ in condition~$(A)$ with the path-connected component~$P(x_0,r)$ as in Remark~\ref{r:OtherAlpha}; however the proof does not work if we instead use~$G(x_0,r)$, since, in this case, the inequality $(\theta'\circ f)\cdot\slo{f}\geq 1$ may not hold on~$G\big(y_{T_i},r-\mssd(x_0,y_{T_i})\big)$.
	\end{remark}

	\section{Convergence of the discrete scheme}\label{sec: convergence of discrete scheme}
	
	This section contains the proof of Theorem \ref{t:conv-disc-new}, which deals with the convergence of proximal point sequences to a global minimiser.
	Our proof is based on adaptation of the arguments in \cite[Thm. 24]{Bol-Dan-Ley-Maz-2010}).
	A key tool is the following result by Ioffe \cite{Ioffe:2000}; see also \cite[Lemma 2.5]{Drusvyatskiy-Ioffe-Lewis:2015}.

	\begin{lemma}[Ioffe's Lemma]\label{lem:Ioffe-lemma}
		Let $g\colon X \to [-\infty,\infty]$ be a lower semicontinuous functional on a complete metric space $(X,\mssd)$. 
		Let $x \in \dom{g}$ and suppose that there are constants $\delta \leq g(x)$
		and $R > 0$ such that
		\footnote{
			Note that the corresponding result in \cite{Drusvyatskiy-Ioffe-Lewis:2015} involves the  closed ball~$\set{y\in X
			: \mssd(x,y)\leq R}$ instead of the open ball $B_R$. It is easy to see that the statements are equivalent, possibly after taking a slightly smaller radius.}
		\begin{align*}
			\slop{g}{u} & \geq v 
			\quad\text{ for all } 
			u \in 
			B_R(x)
			\cap \graf*{\delta < g\le g(x)} 
		\end{align*}
		for some $v > (g(x) - \delta) / R$.
		Then:
		\begin{align}
			\mssd \tparen{x, \graf*{g\leq \delta}} \leq \frac{g(x)-\delta}{v}.
		\end{align}
	\end{lemma}

 Throughout the remainder of this section we impose the following standing assumptions that are in force without further mentioning: 
 \begin{itemize}
	\item $f \colon X\to [0,\infty]$ 
	is a proper and lower semicontinuous
	functional on a complete metric space $(X,\mssd)$;
	\item $x_0 \in \dom{f}$ and $r > 0$ satisfy 
	Condition {$(A')$} for some 
 	parameter function $\theta$;
	\item there exists a time-step
	$\bar{\tau} > 0$ (that will be fixed from now on)
	such that, for all $x \in B_r(x_0)\cap \graf*{f\le f(x_0)} $ and $\tau \in (0, \bar \tau)$, the functional
	\begin{align}
		\label{eq:MMS}
		X \ni y 
			\longmapsto 
			f(y) + \frac{1}{2 \tau} \mssd(x,y)^2
	\end{align}
	has at least one global minimiser. 
	The non-empty set of minimisers will be denoted by $J_\tau (x)$. 

 \end{itemize}

The latter condition is 
 satisfied 
	with $\bar{\tau} = \infty$
	if~$(X,\mssd)$ is
	proper, i.e., if all closed $\mssd$-bounded sets in~$X$ are compact.

	The following result contains some fundamental properties of $J_\tau$, which can be found in \cite{AmbGigSav08} under slightly different assumptions. 
	The same proofs apply to our setting.
	
	\begin{lemma}
		\label{l:length of backward euler step}
		For $x \in B_r(x_0)  \cap \graf*{f\le f(x_0)}$ 
		 the following assertions hold:
		\begin{align}
			\label{eq:f-monotonicity}
			&f(z_0)  \geq f(z_1)
			&&\text{for all } 
			0 < \tau_0 \leq \tau_1 < \bar \tau \text{ and } z_0 \in J_{\tau_0}(x), z_1 \in J_{\tau_1}(x)\semicolon \\
			\label{eq:slope-bound}
			&\mssd(x, z)  \geq \tau \slop{f}{z} 
			&& \text{for all } 
			0 < \tau < \bar \tau \text{ and } z\in J_\tau(x) \semicolon \\
			\label{eq:DeGiorgi}
			& f(z) 
			+ \frac{\mssd(x, z)^2}{2 \tau}
			+ \int_0^\tau \frac{\mssd(x, z_s)^2}{2s^2} \diff s
			= f(x) 
			&&\text{for all } 
			0 < \tau < \bar \tau \text{ and } z\in J_\tau(x), z_s \in J_s(x)\fstop
		\end{align}
	\end{lemma}
	
	\begin{proof}
		Inequality ~\eqref{eq:f-monotonicity} can be found in \cite[Lem.~3.1.2]{AmbGigSav08};
		\eqref{eq:slope-bound} can be found in \cite[Lem.~3.1.3]{AmbGigSav08};
		\eqref{eq:DeGiorgi} can be found in  \cite[Thm.~3.1.4, Eqn.~(3.1.12)]{AmbGigSav08}.
	\end{proof}

	In the following result we consider a slightly more general  notion of proximal point sequences, as we allow the step-size $\tau = \tau_k$ to depend on the step $k$.
	This will be useful in Lemma \ref{lem:concave-theta} below.
	
	\begin{lemma}[Confinement and distance bound]\label{l:dist-bound-disc}
		Let $\seq{y_k}_{k=0}^N$ with $N\in\N$ be a proximal point sequence starting at $x_0$
		with step-sizes $\tau_k \in (0, \bar \tau)$ for $0 \leq k < N$. 
		Then for all $0\leq k \leq \ell \leq N$ we have the distance bound
		\begin{align}\label{eq:dist-bound-discrete}
			\mssd(y_k ,y_\ell) \leq (\theta\circ f)(y_k) - (\theta\circ f)(y_\ell).
		\end{align}
		In particular, $y_k \in B_r(x_0)$ for all $0\leq k\leq N$.
	\end{lemma}

	\begin{proof}
		We will prove that \eqref{eq:dist-bound-discrete} holds for all $0\leq i\leq j \leq N$ by induction on $N$, noting that the case $N = 0$ is trivial. 
		
		We thus suppose that the claim is true for some $N \geq 0$, and
		let $\seq{y_k}_{k=0}^{N+1}$ be a proximal point sequence starting at $x_0$.
		By the induction hypothesis and the triangle inequality it suffices to prove that 
			$\mssd(y_N, y_{N+1})
			\leq 
			(\theta\circ f)(y_N) -
			(\theta\circ f)(y_{N+1})$. 
		If $f(y_N) = 0$, we have $y_N = y_{N+1}$ and the claim follows. We thus assume that $f(y_N) > 0$. 
		
		By Condition {$(A')$} there exists 
			$\epsilon > 0$ 
		such that 
			$(1+\epsilon)(\theta\circ f)(x_0) < r$. 
		We will apply Lemma \ref{lem:Ioffe-lemma} 
		to 
		\begin{align*}
			g = \theta \circ f,
				\quad
			x=y_N, 
				\quad
			\delta = (\theta\circ f)(y_{N+1}),
				\quad
			R = (1+\epsilon)(\theta\circ f)(y_N),
				\quad
			v = 1.
	\end{align*}
	We will show that the assumptions of Lemma \ref{lem:Ioffe-lemma} are satisfied.
		\begin{itemize}
		\item Firstly, we claim that 
		$\slop{g}{u} \geq 1$		
		for $u \in 
			B_R(y_N)
				\cap 
			\graf*{\delta < g\le g(y_N)}$.
		Indeed, by
		the triangle inequality and the induction hypothesis,
		\begin{align*}
			\mssd(x_0, u) 
				& \leq  
			\mssd(x_0, y_N) + \mssd(y_N, u) 
		\\&	\leq 
			\Big( (\theta\circ f)(x_0)  	
			- 
			(\theta\circ f)(y_N)  \Big)
			+ 
			(1+\epsilon)(\theta\circ f)(y_N) 
		\\& =
			(\theta\circ f)(x_0) 
				+ 
			\epsilon (\theta\circ f)(y_N) 
		\leq 
			(1+\epsilon)(\theta\circ f)(x_0) 	
			< 
			r.			
		\end{align*}
		This shows that $u \in B_r(x_0)$.
		Moreover, since $\theta$ is strictly increasing,
		$f(u) \leq f(y_N) \leq f(x_0)$.
		Since furthermore $f(u) > 0$, 
		Condition $(A')$ implies that 
		$\theta'(f(u))
		\slop{f}{u} 
		\geq 
		1$.
		In particular, $\slop{f}{u} > 0$, hence $u$ is not an isolated point. 
		Therefore, 
		Lemma \ref{l: chain rule for slope} 
		yields the desired inequality
		\begin{align*}
			\slop{g}{u} 
			= 
			\theta'(f(u))
			\slop{f}{u} 
			\geq 
			1.
		\end{align*}
		\item
		Secondly, we claim that $g(y_N) - \delta < R$.
		Indeed,  
		\begin{align*}
			g(y_N) - \delta 
			= 
			(\theta \circ f)(y_N) 
				- 
			(\theta \circ f)(y_{N+1}) 
			< 
			(1+\epsilon)(\theta \circ f)(y_N)
			= R.
		\end{align*}
		\end{itemize}
		Using that $\theta$ is strictly increasing, 
		we deduce from Lemma \ref{lem:Ioffe-lemma} that 
		\begin{align*}
			\mssd \big(y_N, \big\{
				f \leq f(y_{N+1})
				\big\}\big) 
			\leq (\theta\circ f)(y_N) 
				- (\theta\circ f)(y_{N+1}).	
		\end{align*}
		This means that for any $\kappa > 0$ there exists 
			$\bar x \in X$ 
		such that 
		\begin{align*}
			f(\bar x) \leq f(y_{N+1})
		\tand
			\mssd( \bar x, y_N ) < 
			(\theta\circ f)(y_N) 
				- (\theta\circ f)(y_{N+1}) + \kappa.
		\end{align*}
		On the other hand, since 
			$y_{N+1} \in J_{\tau_N}(y_N)$, 
		we have 
		\begin{align*}
				\mssd^2(y_{N+1}, y_N)
				\leq 
				2\tau_N \big( f(\bar x) - f(y_{N+1}) \big)
			+ 
				\mssd^2(\bar x, y_N).
		\end{align*}
		As $\kappa > 0$ can be chosen arbitrarily small, a combination of these bounds yields
		\begin{align*}
			\mssd \tparen{y_N, y_{N+1}} 
		 \leq 
		 (\theta\circ f)(y_N) 
			- (\theta\circ f)(y_{N+1}),
		\end{align*}
		which completes the induction step
		and the proof of \eqref{eq:dist-bound-discrete}.

		\smallskip
		The final assertion follows from 
		\eqref{eq:dist-bound-discrete}
		since $\theta(f(x_0)) < r$ 
		by Condition  {$(A')$} .
	\end{proof}
	
	\begin{lemma}
		\label{lem:infinite-pp-seq}
		For any $\tau \in (0,\bar \tau)$
		there exists an \emph{infinite} 
		proximal point sequence 
			$\seq{y_k}_{k \geq 0}$ 
		with $y_0 = x_0$.
	\end{lemma}
\begin{proof}
	It suffices to iteratively construct the sequence $\seq{y_k}_k$ by letting $y_0 \eqdef x_0$ and $y_{k+1}$ be a minimiser in \eqref{eq:MMS}
	with $y_k$ in place of $x$, 
	noting that $y_k\in B_r(x_0)$ by Lemma \ref{l:dist-bound-disc} and $f(y_k)\leq f(x_0)$ by \eqref{eq:DeGiorgi}.
\end{proof}
	
	\begin{lemma}\label{l:limits-of-prox-points}
		Let $(y_k)_{k \geq 0}$ 
		be a proximal point sequence
		and suppose that  
		$y_\infty \coloneqq \lim_{k\to \infty} y_k$
		exists. 
		Then
		\begin{align*}
			f(y_\infty) = \lim_{k\to \infty} f(y_k)
			\tand 
			\lim_{k\to \infty} \slop{f}{y_k} = 0.
		\end{align*}
	\end{lemma}
	\begin{proof}
		The lower semicontinuity of $f$ yields 
			$f(y_\infty) \leq \liminf_{k \to \infty} f(y_k)$.
		On the other hand, since $y_k \in J_\tau (y_{k-1})$ we have 
		\[
		f(y_k) 
			\leq 
			f(y_\infty) 
				+ 
			\frac{1}{2\tau} \mssd(y_{k-1}, y_\infty)^2,
		\]
		hence $\limsup_{k \to \infty} f(y_k) \leq f(y_\infty)$.
		Combining these inequalities, we obtain the first identity.

		As for the second identity,
		note that
		\[0 \leq \slop{f}{y_k} \leq \frac{\mssd(y_{k-1}, y_k)}{\tau}\]
		for $k\geq 1$ by \eqref{eq:slope-bound}. 
		The conclusion follows by letting $k \to \infty$.
	\end{proof}
	
	\begin{proof}[Proof of Theorem \ref{t:conv-disc-new}]
		Fix $\tau \in (0,\bar \tau)$.
		The existence of an infinite 
		proximal point sequence 
			$\seq{y_k}_{k \geq 0}$ 
		with $y_0 = x_0$ and step-size $\tau$
		was proved in Lemma \ref{lem:infinite-pp-seq}.

		Statement \ref{it: confinement discrete scheme} 
		was proved in Lemma \ref{l:dist-bound-disc}.
		
		The distance bound in 
		\ref{it: dist bound discrete}
		was proved for 
		$0\leq i\leq j < \infty$
		in 
		Lemma \ref{l:dist-bound-disc} as well.
		
		To prove \ref{it: convergence discrete scheme}, 
		note that $(\theta(f(y_k)))_k$ is a Cauchy sequence, 
		as it is non-negative and non-increasing.
		Therefore, 
		\ref{it: dist bound discrete} implies the Cauchy property of $\seq{y_k}_k$ and hence the existence of $y_\infty \coloneqq \lim_{k\to \infty} y_k$.
		Using 
			\eqref{eq:dist-bound-discrete}
		and Condition {$(A')$} we infer for $0 \leq i < \infty$
		that 
		\begin{align*}
			\mssd(y_i, y_\infty)
			\leq \liminf_{j \to \infty}
			\mssd(y_i, y_j) 
			\leq 
				(\theta \circ f)(y_i)
			< r.
		\end{align*}
		This show that $y_\infty\in {B_r(x_0)}$
		and the distance bound in \ref{it: dist bound discrete} for $j = \infty$ follows as well.
		To show that $f(y_\infty) = 0$, 
		we may assume that $f(y_k)>0$ for all $k$, since 
		otherwise there is nothing to prove.
		Since $\slop{f}{y_k} \to 0$ as $k \to \infty$ by Lemma \ref{l:limits-of-prox-points} and
		\begin{align*}
			\theta'(f(y_k)) \cdot \slop{f}{y_k} \geq 1
		\end{align*}
		by Condition {$(A')$},
		we infer that 
		$\theta'(f(y_k)) \to \infty$.
		As $\theta'$ is continuous on $(0,\infty)$
		and the sequence $(f(y_k))_k$ is non-increasing, 
		it follows that 
		$f(y_k) \to 0$, hence $f(y_\infty) = 0$ by lower semicontinuity of $f$.
	\end{proof}

	For specific choices of the parameter function $\theta$ it is possible to obtain more explicit estimates on the decay of $f(y_k)$ and $\mssd(y_k,y_\infty)$ as $k \to \infty$. We adapt and refine some arguments from \cite{AttBol2009}, where similar results are proved.

	\begin{lemma}
		\label{lem:concave-theta}
		Let $\seq{y_k}_{k=0}^\infty$ be a proximal point sequence 
		with step-size $\tau \in (0, \bar \tau)$ 
		starting at $x_0$.
		If the parameter function $\theta$ is concave, 
		we have, for all $k \geq 0$,
		\begin{align}
			\label{eq:one-step-decay-concave}
			f(y_k) - f(y_{k+1}) 
					\geq
					\frac{\tau }
					{(\theta'\circ f) (y_{k+1})^2 } \fstop
		\end{align}
	\end{lemma}

	\begin{proof}
		Fix $k \geq 0$
		and take $z_s \in J_s (y_k)$ for $s \in (0, \tau)$.
		Using De Giorgi's formula \eqref{eq:DeGiorgi},
		the inequality \eqref{eq:slope-bound},
		and Condition {$(A')$}, 
 		we obtain
		\begin{align*}
			f(y_k) - f(y_{k+1}) 
				= 
			\frac{\mssd(y_k, y_{k+1})^2}{2 \tau}
				+ 
			\int_0^\tau \frac{\mssd(y_k, z_s)^2}{2s^2} \diff s
			&
				\geq 
			\frac{\tau}{2} \slop{f}{y_{k+1}}^2
				+ 
			\frac12 \int_0^\tau {\slop{f}{z_s}^2} \diff s
			\\ & \geq 
			\frac{\tau }{2 (\theta'\circ f) (y_{k+1})^2}
				+ 
			\int_0^\tau \frac{1}{2 (\theta'\circ f) (z_s)^2} \diff s
			\fstop
		\end{align*}
		Note that Condition {$(A')$} can be applied to $z_s$ since $f(z_s) \leq f(y_k) \leq f(x_0)$ and
		Lemma \ref{l:dist-bound-disc} implies that 
		$z_s \in B_r(x_0)$.
		Since $\theta$ is concave and $f(z_s) \geq f(y_{k+1})$ by \eqref{eq:f-monotonicity}, the result follows.
	\end{proof}

	Note that a weaker decay estimate with an additional factor of $1/2$ in the right-hand side of \eqref{eq:one-step-decay-concave}
	can be obtained without using De Giorgi's identity \eqref{eq:DeGiorgi}. 
	
	\begin{corollary}\label{cor:conv-rate-power-dis}
		Suppose that the parameter function $\theta$ is given by
		$\theta(u)
			=
		\frac{c}{\gamma} u^\gamma$
		for some~$c>0$ and~$\gamma \in (0,1]$. 
		Let $\seq{y_k}_{k\geq 0}$ be a proximal point sequence starting at $x_0$
		with step-size $\tau \in (0, \bar \tau)$,
		and set $y_\infty := \lim_{k\to \infty} y_k$.
		The following assertions hold:
		\begin{enumerate}[$(i)$]
			\item \label{it:t:disc-rate-finite} If $\gamma = 1$, then $y_k = y_\infty$ and $f(y_k) = 0$ for all  
			$k \geq \lceil cr/\tau \rceil$.
			\item \label{it:t:disc-rate-fast} If $\frac{1}{2} < \gamma <1$, then, for $k \geq 0$,
			\begin{align*}				
				f(y_k) & 
					\leq 
					\tond*{1+\frac{\tau}{c^2}f(x_0)^{ 1-2\gamma }}^{-k}f(x_0) \comma
				\\
				\mssd(y_k,y_\infty) 
					& \leq 
					\frac{c}{\gamma} f(y_k)^\gamma 
					= 
					\mathcal{O} \tond*{\tond*{1+\frac{\tau}{c^2}f(x_0)^{ 1-2\gamma \EEE }}^{-k\gamma}} \comma		
			\end{align*}
			and, for $k \geq k_0
			\coloneqq
			\log_{1+\tilde\alpha}
				\tond*{ 2\tilde\alpha^{-{1}/{(2\gamma-1)}} }
			$
			with
			$\tilde \alpha \coloneqq \frac{\tau}{c^2} f(x_0)^{1-2\gamma}$,
			\begin{align*}				
				f(y_k) & \leq 
				\Big(\frac{\tau}{c^2}\Big)^\frac{1}{2\gamma-1} 
				2^{-(2-2\gamma)^{-(k-k_0)}} \comma
				\\
				\mssd(y_k,y_\infty) 
					& \leq 
					\frac{c}{\gamma} f(y_k)^\gamma 
					= 
				\mathcal{O}\tond*{ {2^{-\gamma \tond*{2-2\gamma}^{-(k-k_0)}}}}\fstop
			\end{align*}				
			
			\item \label{it:t:disc-rate-exponential} If $\gamma = \frac{1}{2}$ then, for $k \geq 0$,
			\begin{align*}				
				f(y_k) & \leq \tond*{1+\frac{\tau}{c^2}}^{-k}f(x_0) \comma
				\\
				\mssd(y_k,y_\infty) 
				& \leq 2c \sqrt{f (y_k) }
				\leq 2c \tond*{1+\frac{\tau}{c^2}}^{-k/2}
				\sqrt{f(x_0)} \fstop
			\end{align*}
			
			\item \label{it:t:disc-rate-polynomial} If $0<\gamma<\frac{1}{2}$ 
			then, for $k \geq 0$,
			\begin{align*}				
				f(y_k) 
					& \leq 
				 		\Big( f(x_0)^{-(1-2\gamma)} + C_1 k \Big) 
								^{-\frac{1}{1-2\gamma}} 
					= 
				\mathcal{O} \big(
						k^{-\frac{1}{1-2\gamma}}\big),
				\\
				\mssd(y_k,y_\infty) 
				& \leq 
					\frac{c}{\gamma} f(y_k)^\gamma 
				 	= 
				 \mathcal{O} \big({k^{-\frac{\gamma}{1-2\gamma}}}\big),
			\end{align*}		
			where
			$ 
			C_1 = \sup_{R > 1}\min 
			\big\{
				\tfrac{\tau(1-2\gamma)}{c^2R}
				, 
				\big({R^\frac{1-2\gamma}{2-2\gamma}-1}\big) f(x_0)^{2\gamma-1}
				\big\}$.
		\end{enumerate}
	\end{corollary}

	\begin{proof}
		$(i)$: \
		Suppose that $f(y_K) > 0$ for some $K \geq 0$. 
		Condition $(A')$ yields
			$\slop{f}{y_k} \geq \frac{1}{c}$ 
		for all $0 \leq k \leq K$, 
		hence $\mssd( y_{k}, y_{k+1} ) \geq \frac{\tau}{c}$
		for all $0 \leq k \leq K-1$ by \eqref{eq:slope-bound}. 
		Using
		Condition $(A')$ and Lemma \ref{l:dist-bound-disc} we infer that
		\begin{align*}
			r 
				>
			c(f(x_0) - f(y_K))
			= c \sum_{k=0}^{K-1} 
					\Big( f(y_k) - f(y_{k+1}) \Big)
			\geq 
			\sum_{k=0}^{K-1} 
					d(y_k,y_{k+1})
			\geq \frac{K \tau}{c}, 
		\end{align*}
		hence $K < cr/\tau$. It follows that $f(y_k) = 0$ 
		and therefore $y_k = y_\infty$  for all
		$k \geq \lceil cr/\tau \rceil$.

		$(ii) - (iv)$: \
		Write $f_k \coloneqq f(y_k)$.
		From $\eqref{eq:one-step-decay-concave}$ we deduce the recursive inequality 
		\[
			f_{k-1} - f_k \geq \frac{\tau}{c^2} f_k^{2-2\gamma} \fstop
		\]
		This inequality yields the decay of $f_k$ using Lemma \ref{l:rec-ineq-app} with 
		$\alpha = {\tau}/{c^2}$ and $\delta = 2-2\gamma$. 
		The decay estimates for 
			$d_k \coloneqq  \mssd(y_k,y_\infty)$ 
		follow 
		from the bounds for $f_k$ combined with 
		Theorem \ref{t:conv-disc-new}\ref{it: dist bound discrete}.
	\end{proof}

	\begin{remark}
		In general --- unlike in the continuous setting of Corollary \ref{cor:conv-rate-gf-special} --- the  discrete scheme does not converge in a finite number of steps when $\frac{1}{2}<\gamma<1$. 
		For suitable $f\in C^1(\R^n)$, this is easy to deduce from the equivalent formulation in \eqref{eq:back-euler}. 
		However, 
		the corollary above shows that the rate of convergence is (asymptotically) faster than exponential. 
	\end{remark}
	
	\begin{ese}[Non-uniqueness revisited]
		\label{ex:non-uniqueness-MMS}
		Let us consider again the function 
		$f$ 
		defined in \eqref{eq:f-non-unique} (see Fig.~\ref{fig:NonUniqueness:1})
		and let $\tau > 0$. 
		For any $x_0 \neq 0$, the resolvent $J_\tau(x_0)$ is single-valued.
		However, if $x_0 = 0$, the resolvent $J_\tau(x_0)$ contains two elements, say $x_\tau$ and $- x_\tau$. 
		Consequently, there are two distinct proximal point sequences starting at $x_0 = 0$; once $x_\tau^1 \in J_\tau(x_0)$ is selected, the rest of the sequence is determined. 
		Corollary~\ref{cor-disc-square-root} implies the exponential convergence for both  of these sequences.
	\end{ese}

	\appendix
	
	\section{A nonsmooth chain rule}
	
	In practice, it can be difficult to compute the slope of a non differentiable function, because tools such as the chain rule are missing.
	We have however the following basic substitute Lemma.
	\begin{lemma}\label{l: chain rule for slope}
		Let $f\colon X \to (-\infty, \infty]$ be proper and lower semicontinuous and let
		$g\colon \R\to \R$ be lower semicontinuous and non-decreasing.
		Then $g\circ f\colon X\to (-\infty, \infty]$ is proper and lower semicontinuous.
		Furthermore, if $x\in \dom{f}$ is not isolated and is such that there exists the left derivative~$\partial_- g\tond*{f(x)} \geq 0$ of $g$ at $f(x)$, then
		\begin{equation}\label{e: chain rule for the slope}
			\slo{(g\circ f)}(x) = \partial_- g \tond*{f(x)} \cdot \slo{f}(x) \fstop
		\end{equation}
	\end{lemma}
	\begin{proof}
		The lower semicontinuity of $g\circ f$ is well known, so we only prove \eqref{e: chain rule for the slope}.
		Set
		\[
		S \coloneqq \set{\tond*{y_n}_n \subset X : y_n \to x\comma y_n \neq x}.
		\]
		Since $x$ is not isolated, $S\neq \varnothing$.
		For a function $h\colon X \to (-\infty, \infty]$ we have
		\begin{equation}\label{e: sequential characterization of limsup}
			\limsup_{y\to x} h (y) = \sup \set{ \limsup_{n\to \infty} \:h(y_n) : \tond*{y_n}_n \in S} = \max \set{ \limsup_{n\to \infty} \:h(y_n) : \tond*{y_n}_n \in S}.
		\end{equation}
		We first prove that
		\begin{equation}\label{eq:ChainRule:+}
			\slo{(g\circ f)}(x) \leq \partial_- g \tond*{f(x)} \cdot \slo{f}(x) \fstop
		\end{equation}
		Let $\seq{y_n}_n \in S$.
		We need to show that
		\[
		\limsup_{n\to \infty} \frac{\quadr*{(g\circ f)(y_n)-(g \circ f)(x)}_- }{\mssd(y_n,x)} \leq \partial_- g \tond*{f(x)} \cdot \slo{f}(x) \fstop
		\]
		Observe that if $f(y_n) \geq f(x)$ then, since $g$ is non decreasing,
		\[
		\frac{\quadr*{(g\circ f)(y_n)-(g \circ f)(x)}_- }{\mssd(y_n,x)} = 0\fstop
		\]
		Therefore without loss of generality (by changing sequence and/or restricting to a subsequence) we can assume that $f(y_n) <f(x)$.
		Then, passing to the limit superior as~$n\to\infty$ we see that~$\limsup_n f(y_n)\leq f(x)$.
		By lower semicontinuity of~$f$ we have as well that~$\liminf_n f(y_n)\geq f(x)$, thus there exists~$\lim_n f(y_n)=f(x)$.
		Noting also that
		\begin{equation}\label{eq:ChainRule:1}
			\begin{aligned}
				\frac{\quadr*{(g\circ f)(y_n)-(g \circ f)(x)}_- }{\mssd(y_n,x)} & = \frac{\tquadre{(g\circ f)(y_n)-(g \circ f)(x)}_-}{\tquadre{f(y_n)-f(x)}_-} \cdot \frac{\tquadre{f(y_n)-f(x)}_-}{\mssd(y_n, x)} 
				\\
				& = 
				\frac{g(f(x)) - g(f(y_n))}{f(x) - f(y_n)} \cdot \frac{\quadr*{f(y_n)-f(x)}_-}{\mssd(y_n, x)} 
			\end{aligned}
		\end{equation}
		we conclude that
		\begin{align*}
			\limsup_{n\to \infty} 	\frac{\tquadre{(g\circ f)(y_n)-(g \circ f)(x)}_- }{\mssd(y_n,x)} \leq  \partial_- g \tond*{f(x)} \cdot \slo{f}(x)
		\end{align*}
		as desired.
		
		\medskip
		
		We now prove the converse inequality:
		\begin{equation}\label{eq:ChainRule:-}
			\slo{(g\circ f)}{(x)} \geq \partial_- g \tond*{f(x)} \cdot \slo{f}(x) \fstop
		\end{equation}
		If $\slo{f}(x) = 0$ the claim is trivial, so we assume now that $\slo{f}(x)>0$. By \eqref{e: sequential characterization of limsup} there exists $\seq{y_n}_n \in S$ such that
		\[
		\limsup_{n\to \infty} \frac{\tquadre{f(y_n)-f(x)}_-}{\mssd(y_n, x)} = \slo{f}(x)>0 \fstop
		\]
		By restricting to a subsequence we can assume that $f(y_n)<f(x)$.
		Arguing as before we have the existence of $\lim_n f(y_n)= f(x)$ and that~\eqref{eq:ChainRule:1} holds.
		Taking the limit superior as $n\to \infty$ in~\eqref{eq:ChainRule:1} gives
		\[
		\slo{(g\circ f)}(x) \geq	\limsup_{n\to \infty} 	\frac{\quadr*{(g\circ f)(y_n)-(g \circ f)(x)}_- }{\mssd(y_n,x)} \geq  \partial_- g \tond*{f(x)} \cdot \slo{f}(x)
		\]
		as desired.
		Combing~\eqref{eq:ChainRule:-} with the opposite inequality~\eqref{eq:ChainRule:+} yields the assertion.
	\end{proof}

	\section{Estimating recursive inequalities}
	
	The following lemma contains some estimates that are used  in the proof of Corollary \ref{cor:conv-rate-power-dis}.

	\begin{lemma}\label{l:rec-ineq-app}
		Let $\seq{f_k}_{k\in \N}$ be a sequence of 
		non-negative real numbers with $f_0 > 0$
		and suppose that for some $\alpha, \delta > 0$ 
		the recursive relation 
		\begin{align} \label{eq:rec-ineq-app}
			f_{k-1} - f_k \geq  \alpha  f_k^\delta 
		\end{align}
		holds for all $k \geq 1$.
		Then, for all $k \geq 0$:
				\begin{subnumcases}{f_k \leq}
					\label{eq:f3}
					\big(f_0^{-(\delta-1)}+Ck \big)
				^{-1/(\delta-1)} 
				& if
					$\delta > 1 \comma$ 
					\\
					\label{eq:f2}
					(1+\alpha)^{-k} f_0 & if
					$\delta = 1 \comma$ 
					\\
					\label{eq:f1}
					\tond*{1+\tilde \alpha}^{-k} f_0 
						& 
					if
					$\delta < 1 \comma$
				  \end{subnumcases}
		where 
		$\tilde \alpha = {\alpha}/ {f_0^{1-\delta}}$
		and
		\[C \coloneqq \sup_{R>1} \min \graf*{\frac{\alpha(\delta-1)}{R}, \tond*{R^{(\delta-1)/\delta}-1} f_0^{1-\delta}} \fstop\]
		Furthermore, if $\delta < 1$, we also have
		\begin{align}
			\label{eq:f4}
			f_k \leq \alpha^{{1}/{(1-\delta)}} 
				\cdot {2^{-\delta^{-(k-k_0)}}}
				\quad
				\text{for  \ } 
				 k \geq k_0
				 	\coloneqq 
				\log_{1+\tilde\alpha}
				\tond*{ 2\tilde\alpha^{-{1}/{(1-\delta)}} }
				\fstop
		\end{align}

	\end{lemma}
	
	\begin{proof}
		Note first that in all cases the sequence $\seq{f_k}_k$ is non-increasing.
		
		\smallskip
		\noindent
		\eqref{eq:f3}: \
		We follow the arguments of \cite{AttBol2009}. 
		Fix $R\in(1,\infty)$ and consider the concave function $H (s)   = \frac{1}{1-\delta} s^{1-\delta}$ and its derivative $h(s) = H'(s)= s^{-\delta}$.
		Observe that \eqref{eq:rec-ineq-app} can be equivalently written as 
		\[
		\alpha \leq  \tond*{f_{k-1} - f_{k}}h(f_k) \fstop
		\]				
		
		Suppose first that $h(f_k)\leq R h(f_{k-1})$.
		Using the concavity of $H$ we obtain			
		\begin{align*}
			\alpha
			&\leq 
				\tond*{f_{k-1} - f_{k}}h(f_k)
			 \leq  
			 	R
			 	\tond*{f_{k-1} - f_{k}}h(f_{k-1})
			 \leq 
			 	R
			 	\tond*{H(f_{k-1})-H(f_k)}
			 = 
			 	\frac{R}{\delta-1} 
				\tond*{f_k^{1-\delta} - f_{k-1}^{1-\delta}}.
		\end{align*}
		Writing 
			$C_1(R) 
			\coloneqq 
			{\alpha(\delta-1)}/{R} 
			>0
			$,
		this shows that 
			\begin{align*}
				f_{k}^{1-\delta}-f_{k-1}^{1-\delta} \geq C_1(R) \fstop
			\end{align*}
		
			Suppose next that instead $h(f_k)>R h(f_{k-1})$. 
		Raising this inequality to the power $\frac{\delta-1}{\delta}$ we obtain
		\[
		f_k^{1-\delta} > R^\frac{\delta-1}{\delta}
				f_{k-1}^{1-\delta}
		\]
		and hence, since $(f_k)_k$ is non-increasing,
		\begin{align*}
			f_{k}^{1-\delta} - f_{k-1}^{1-\delta} &\geq \tond*{R^\frac{\delta-1}{\delta}-1}f_{k-1}^{1-\delta}
			 \geq \tond*{R^\frac{\delta-1}{\delta}-1} f_0^{1-\delta}
			 \eqqcolon C_2 (R) >0. 
		\end{align*}
		
		Finally, defining $C\coloneqq \sup_{R\in (1,\infty)} \min\{ C_1(R), C_2(R)\} > 0$, 
		the above  inequalities combined yield
		\begin{align*}
			f_{k}^{1-\delta}-f_{k-1}^{1-\delta} \geq C.
		\end{align*}
		Evaluating a telescopic sum, we obtain
		$
			f_k^{1-\delta} - f_0^{1-\delta} \geq Ck
		$.
		Rearranging terms we obtain the desired inequality
		\eqref{eq:f3}.

		\smallskip
		\noindent
		\eqref{eq:f2}: \
		  This is straightforward.

		\smallskip
		\noindent		
		\eqref{eq:f1}: \ 
		Suppose without loss of generality that $f_k > 0$.
		From \eqref{eq:rec-ineq-app} it follows that
		\begin{align*}
			\frac{f_{k-1}}{f_k} 
				\geq 1 + \frac{\alpha}{f_k^{1-\delta}} 
				\geq 1 + \frac{\alpha}{f_0^{1-\delta}}
				=	 1 + \tilde \alpha,
		\end{align*}
		from which we deduce \eqref{eq:f1}.

		\smallskip
		\noindent
		\eqref{eq:f4}: \
		Writing
			$\tilde{f}_k \coloneqq 
			{\alpha^{-1/(1-\delta)}} f_k$
		we note that \eqref{eq:rec-ineq-app} implies
		$
			\tilde{f}_k^\delta \leq \tilde{f}_{k-1}
		$
		and therefore
		\begin{align*}
			\tilde{f_k} 
				\leq 
			\big(\tilde{f}_{k_0}\big)^
				{{\delta^{-(k - k_0)}}} \fstop
		\end{align*}
		Moreover, \eqref{eq:f1} and the definition of $k_0$ yield
			$
			\tilde f_{k_0} 
				\leq 
			(1 + \tilde \alpha)^{-k_0} \tilde f_0 
				=
			\frac12
			$.
		Combining these estimates, we obtain the desired estimate
		$\tilde{f_k} \leq 2^
		{-{\delta^{-(k - k_0)}}}$.
		\end{proof}

	{\small
	\bibliographystyle{abbrv}
%	\bibliography{GradFlows.bib}

	}

\end{document}